\documentclass[11pt,reqno]{amsart}

\usepackage[T1]{fontenc}
\usepackage{amsmath,amsthm,amssymb}
\usepackage{geometry}
\usepackage{xcolor}
\usepackage[hidelinks]{hyperref}
\IfFileExists{mathrsfs.sty}{\usepackage{mathrsfs}}{\newcommand{\mathscr}[1]{\mathcal{#1}}}

\allowdisplaybreaks
\numberwithin{equation}{section}

\newtheorem{theorem}{Theorem}[section]
\newtheorem{proposition}[theorem]{Proposition}
\newtheorem{lemma}[theorem]{Lemma}
\newtheorem{corollary}[theorem]{Corollary}

\theoremstyle{definition}

\newtheorem{remark}[theorem]{Remark}

\newcommand{\R}{\mathbb R}

\newcommand{\dd}{\,\mathrm d}

\newcommand{\weakto}{\rightharpoonup}

\newcommand{\ind}{\mathbf 1}
\newcommand{\sgn}{\operatorname{sgn}}

\makeatletter
\def\@settitle{\begin{center}\baselineskip=18pt\relax
  {\LARGE\sffamily\bfseries\@title\par}\end{center}}
\makeatother

\title[Finite-data inverse nodal optimization]{Finite-data inverse nodal optimization in angular-momentum sectors of Schr\"{o}dinger operators}

\author{Xijun Deng, Zhisu Liu, Yonghui Xia}

\address[X. J. Deng]{\newline\indent School of Mathematics,
\newline\indent
Hubei University of Automotive Technology,
Shiyan, Hubei, 442002, P. R. China}
\email{\href{mailto:xijundeng@yeah.net}{xijundeng@yeah.net}}

\address[Z. S. Liu]{School of Mathematics and Physics, China University of Geosciences,
\newline\indent
Wuhan, Hubei 430074, P. R. China\\
\newline\indent
Institute for Advanced Marine Research, China University of Geosciences,
\newline\indent
Guangzhou 511462, P. R. China}
\email{\href{liuzhisu@cug.edu.cn}{liuzhisu@cug.edu.cn}}

\address[Y. H. Xia]{School of Mathematics, Foshan University, Foshan, 528133, P. R. China}
\email{\href{yhxia@zjnu.cn}{yhxia@zjnu.cn}}

\thanks{Liu was supported by the National Natural Science Foundation
of China (No.12571188), and Xia was supported by the National Natural Science Foundation
of China (No.12571165) and the Natural Science Foundation of Guangdong Province (2026A1515011073).}

\date{}


\begin{document}

\begin{abstract}
In this paper, we study a finite-data inverse nodal optimization problem for radial
Schr\"odinger operators
$$
 H_q:=-\Delta+q(|x|),\qquad u|_{\partial B_R}=0,
 \qquad d\geq2,
$$
on the ball $B_R\subset\mathbb R^d$, in an arbitrary fixed angular-momentum
sector.  The analysis is built directly at the Friedrichs
endpoint and in the physical weighted space $L_d^p$, $p>d/2$, so that the
singular radial geometry is retained rather than replaced by a regular
one-dimensional model.

The main purpose of this paper is to provide \emph{a singular Friedrichs finite-data variational framework} valid in every
angular-momentum sector, thereby extending the existing finite-data variational theories
concerning either regular one-dimensional operators or the radial sector $\ell=0$. By means of a Volterra
representation of the Friedrichs branch, we prove weak continuity and continuous
Fr\'{e}chet differentiability of nodal radii, exact realization of compatible same-mode nodal data, existence of optimal
potentials, and finite-codimensional constraint geometry. The same framework also
incorporates mixed angular-momentum and spectral--nodal observations through finite-dimensional transversality.

Remarkably, \emph{a global uniqueness theorem} is established for inward displacements of
the unique interior node of the second mode in the radial sector \(\ell=0\).
 For a constant reference potential and \(p>(d+2)/2\), every such displacement
 admits a unique global optimizer. Unlike local inverse-mapping or one-dimensional
 integrability arguments, the proof first selects the admissible critical sign globally
 and then reduces every minimizer to a scalar mass-balance equation between a focusing ball
 branch and a logistic annulus branch. The strict opposite monotonicity of the two weighted
 masses makes the balance parameter unique,
 providing a global rigidity mechanism over the entire inward-displacement regime.
\end{abstract}

\maketitle

\section{Introduction}
Radial Schr\"odinger operators arise naturally in quantum mechanics for
particles moving in central or spherically confining fields.  After separation
of variables, angular momentum produces an inverse-square centrifugal term,
while the zeros of a radial eigenfunction correspond to spherical nodal
shells and therefore provide spatial information about the underlying
interaction or confinement potential; see, for example,
\cite{LandauLifshitz1977,Teschl2014}.  This structure also appears in
effective-mass models of spherical semiconductor quantum dots, where the shape
of the radial confinement potential has a decisive influence on energy levels
and wave functions \cite{GharaatiKhordad2010}.  From the inverse-quantum
viewpoint, recovering an unknown interaction from spectral or wave information
is a fundamental problem \cite{ChadanSabatier1989}.  These considerations
motivate a finite-data formulation: rather than assuming an asymptotically
rich or dense set of nodal measurements, one seeks a physically admissible
potential compatible with finitely many observations and, among all such
potentials, selects the one closest to a prescribed reference profile.

Motivated by these considerations, we study a finite-data inverse nodal
optimization problem for radial Schr\"odinger operators on higher-dimensional
balls.  Given finitely many nodal radii, possibly together with one eigenvalue,
we look for a radial potential which realizes the prescribed data and has
minimal distance from a given reference potential in the physical radial
$L^p(B_R)$ metric.  We prove existence of optimal potentials, establish
Fr\'echet differentiability and explicit gradient formulas for the nodal maps,
and obtain local minimum-norm reconstruction and uniqueness results.  We also
treat mixed spectral--nodal data and, in a particular radial configuration,
prove global uniqueness of the optimizer.

We consider the Dirichlet Schr\"odinger operator
\begin{equation}\label{eq:full-operator}
 H_q=-\Delta+q(|x|),\qquad u|_{\partial B_R}=0,
 \qquad d\geq2,
\end{equation}
on the ball $B_R\subset\mathbb R^d$, with a radial potential $q$.  Since
$H_q$ commutes with rotations, it decomposes into angular-momentum sectors.
If $Y_{\ell,k}$ is a spherical harmonic of degree $\ell\in\mathbb N_0$, then
a separated eigenfunction $u(r)Y_{\ell,k}(\omega)$ has radial profile satisfying
\begin{equation}\label{eq:sector-equation-intro}
 -u''-\frac{d-1}{r}u'
 +\frac{\ell(\ell+d-2)}{r^2}u+q(r)u=\lambda u,
 \qquad 0<r<R,
\end{equation}
with the Friedrichs regularity condition at $r=0$ and $u(R)=0$.  Under the
Liouville transformation $v=r^{(d-1)/2}u$, this becomes the singular
one-dimensional equation
\begin{equation}\label{eq:sector-liouville-intro}
 -v''+\left(\frac{\nu_\ell^2-\frac14}{r^2}+q(r)\right)v=\lambda v,
 \qquad \nu_\ell:=\ell+\frac{d-2}{2}.
\end{equation}
Thus the radial nodal problem naturally leads to a singular Sturm--Liouville
equation.  The singularity at the origin has to be taken into account in the
spectral and perturbative analysis.  In a regular
Sturm--Liouville problem on a compact interval both endpoints admit ordinary
trace boundary conditions and the solution map can be initialized by standard
Cauchy data.  At $r=0$, by contrast, the coefficients in
\eqref{eq:sector-equation-intro}--\eqref{eq:sector-liouville-intro} are
singular; the admissible branch is selected by membership in the Friedrichs
form domain.  For $\ell>0$ the centrifugal term
$\ell(\ell+d-2)r^{-2}$ is explicit, while after the Liouville transform the
coefficient $(\nu_\ell^2-1/4)r^{-2}$ also records the endpoint singularity
(and is critical in the case $d=2$, $\ell=0$).  Consequently we cannot
simply prescribe $v(0)$ and $v'(0)$ and invoke the regular ODE dependence
machinery used in the classical theory.

The development of the subject may be summarized in three stages.  First,
classical inverse nodal theory established that nodal points can determine a
one-dimensional Sturm--Liouville potential.  The early uniqueness theorem of
McLaughlin \cite{McLaughlin1988} was followed by the reconstruction theory of
Hald and McLaughlin \cite{HaldMcLaughlin1989,H-M} and by further uniqueness
and reconstruction results of Yang and of Guo and Wei
\cite{Yang1997,yang2,guo-wei}.  Extensions include discontinuous and
quasilinear problems \cite{YANGCF,w-y,P-S} and nonlinear boundary-value
settings \cite{V-I2019}.

Second, the dependence of nodes on the potential and the effect of incomplete
or noisy nodal information became central.  Weak-topology continuity for ODE
systems was developed by Zhang \cite{zhang}, and Chen, Cheng, and Law
\cite{Chen-Cheng11} reconstructed a regular one-dimensional potential from
the zeros of one eigenfunction by three procedures, including Tikhonov
regularization, with convergence and error estimates.  More recently, Guo and
Zhang \cite{GuoZhang2022} proved complete continuity and variational formulas
for Sturm--Liouville nodes, while Chu, Meng, Wang, and Zhang
\cite{ChuMengWangZhang2025} obtained corresponding differentiability results
for one-dimensional nonlinear operators.  These results supply the compactness
and linearization needed for optimization under nodal constraints.

Third, inverse nodal questions were formulated as variational optimization
problems.  Guo and Zhang \cite{G-Z2} optimized the location of the unique node
of the second Dirichlet eigenfunction.  Chu, Meng, Wang, and Zhang
\cite{ChuMengWangZhang2024} characterized optimal nodal locations under an
$L^p$ constraint on the potential.  He, Wu, Xia, and Zhang
\cite{HeWuXiaZhang2025} then introduced a nearest-target formulation for
finitely many nodal observations in the regular one-dimensional problem.  Its
optimality conditions are nonlinear Schr\"odinger equations; for a constant
target the critical system is completely integrable, reduces to three
characteristic parameters, and yields uniqueness for $p>3/2$.  More recently,
Chu, Meng, and Xie \cite{ChuMengXie2026} derived sharp $L^1$ lower bounds when
the unique node of the second eigenfunction is fixed.

Singular inverse nodal problems form a parallel line of development.
Regular-singular, Bessel-type, and discontinuous problems were studied in
\cite{PanakhovKoyunbakan2006,KoyunbakanPanakhov2006,KoyunbakanPanakhov2007},
while inverse spectral theory for radial and Bessel-type Schr\"odinger
operators was developed in
\cite{AlbeverioHrynivMykytyuk2007,Serier2007,KostenkoSakhnovichTeschl2010,
KostenkoTeschl2011,XuYangBondarenko2023}; recent work treats interaction
potentials and finite-data numerical
reconstruction \cite{ArslantasDurakAmirov2026,JiangXuYang2026}.  These results
are primarily based on nodal asymptotics or direct reconstruction.  By
contrast, the present problem imposes finitely many observations as
constraints on an infinite-dimensional class and selects the compatible
potential nearest to a prescribed reference profile.

Several additional difficulties arise when one passes from the regular
one-dimensional setting to radial Schr\"odinger operators on
$B_R\subset\mathbb R^d$.

First, the
radial reduction carries the physical weight $r^{d-1}\,dr$, so both the
optimization metric and the duality structure differ from their unweighted
one-dimensional counterparts.  In particular, weak compactness, form bounds,
and the representation of nodal derivatives must be proved in $L_d^p$ and its
weighted dual rather than in the usual unweighted spaces.

Second, the origin
is a singular endpoint.  In particular, for $\ell>0$, the inverse-square centrifugal term prevents \(r=0\) from being
treated as an ordinary regular endpoint, so the admissible branch must be selected through the
Friedrichs form domain rather than by prescribing standard Cauchy data.


Third, nodal observations may come from several eigenmodes
or angular-momentum sectors, where the automatic one-mode independence
mechanism no longer applies and a transversality theory is required.  Finally,
finite spectral--nodal constraints lead naturally to a mixed inverse problem
which is not covered by the standard inverse nodal framework.

Consequently, the regular one-dimensional finite-data theory cannot be
transferred by a formal change of notation: the singular endpoint must be
incorporated simultaneously into the spectral realization, compactness
theory, Fr\'echet linearization, boundary-term analysis, and nonlinear
optimality system. The recent work of Cheng, He, Wang, and Xia
\cite{ChengHeWangXia2026} developed the finite variational inverse nodal
problem for the radial branch $\ell=0$ on a ball.  In the physical weighted
space it establishes nodal $C^1$ regularity and weak continuity, exact
realization, finite-codimensional constraint geometry, existence of
nearest-target potentials, critical equations, and cross-node mass balance;
it also gives a local benchmark uniqueness result in the Hilbert case under a
$C^2$ hypothesis.

At the level of formulation, the earlier radial-sector theory is recovered from the present framework
by setting \(\ell=0\). The extension to \(\ell\ge1\), however, is analytically nonformal because
the centrifugal inverse-square term changes the
Friedrichs endpoint analysis, the linearized nodal calculus, and the variational identities.
More precisely,
in the
radial sector the profile equation contains no centrifugal inverse-square
term, and its energy identity retains sign-definite radial dissipation, as said before.  For
$\ell\ge1$, however, the term
$\ell(\ell+d-2)r^{-2}$ changes both the endpoint and the nonlinear structure:
finite form energy selects the branch $u=r^\ell V$, the potential derivative
must remain on the same Friedrichs branch, and the endpoint Wronskian can be
discarded only after establishing new Volterra estimates in the effective
dimension $d+2\ell$.  The same centrifugal term persists in the
Euler--Lagrange and energy identities, where it destroys the monotonicity
underlying the radial global argument.  Thus the general fixed-sector theory
requires additional endpoint calculus, compactness and linearization
analysis, and variational identities; the radial $\ell=0$ results cannot be
transferred directly.

The present paper advances that foundation in two directions. First, it establishes a singular Friedrichs
finite-data variational framework, within which it develops the Friedrichs nodal calculus for every fixed
angular-momentum sector, treats mixed-sector observations under transversality, and proves automatic
transversality for same-mode multi-node and spectral-nodal data. Second,
it resolves the spectral-matching problem globally for every inward displacement of
the unique node of the second radial mode.  These extensions, rather than
the introduction of the radial finite-data model itself, constitute the main
novel content of the paper.

We now state the problem precisely.  Throughout
Sections~\ref{sec:critical}--\ref{sec:multi-node}, we fix one angular momentum
$\ell\in\mathbb N_0$ and set
\begin{equation}\label{eq:fixed-sector-parameters}
 \gamma:=\ell(\ell+d-2),\qquad
 \nu:=\ell+\frac{d-2}{2}.
\end{equation}
For notational economy, $\lambda_m(q)$, $E_m(\cdot;q)$,
$T_{i,m}(q)$, and $g_{i,m}(\cdot;q)$ denote the eigenvalue,
normalized eigenfunction, $i$th interior node, and nodal gradient in this
fixed sector.  Full superscripts $\lambda_{\ell,m}$ and
$T_{i,m}^{(\ell)}$ are restored in Section~\ref{sec:mixed-data}, where
different sectors are used simultaneously.  Thus the $\ell=0$ radial profile
is included throughout as a special case rather than treated by a separate
parallel theory.

Write
\begin{equation}\label{eq:weighted-spaces}
 \dd\mu_d(r)=r^{d-1}\dd r,
 \qquad L_d^p:=L^p((0,R),\dd\mu_d),
 \qquad p'=\frac{p}{p-1}.
\end{equation}
Up to the constant factor $|\mathbb S^{d-1}|^{1/p}$, this is exactly the
physical radial $L^p(B_R)$ norm.  For fixed $\ell$, let
\[
 \lambda_1(q)<\lambda_2(q)<\cdots,
\]
and let the $m$th eigenfunction have simple interior zeros
\[
 0<T_{1,m}(q)<\cdots<T_{m-1,m}(q)<R.
\]
Given one observed node $T_*\in(0,R)$ and a reference potential
$q_0\in L_d^p$, the basic inverse optimization problem is
\begin{equation}\label{eq:RINP}
 \min\left\{\|q-q_0\|_{L_d^p}:
 T_{i,m}(q)=T_*,\ q\in L_d^p\right\},
\end{equation}
with admissible set
\begin{equation}\label{eq:admissible-set}
 \mathcal S_{T_*}:=\{q\in L_d^p:T_{i,m}(q)=T_*\}.
\end{equation}

The contributions of the paper are organized around two main points.

The first contribution is a singular Friedrichs finite-data variational
framework for an arbitrary fixed angular-momentum sector. The factorization $u=r^\ell V$
reveals the effective dimension $d+2\ell$, preserves the linearized Friedrichs branch, and
eliminates the endpoint Wronskian term. This framework yields weak continuity and continuous Fr'echet
differentiability of nodal radii, exact realization of compatible same-mode nodal data, existence of optimal
potentials, and finite-codimensional constraint geometry. In particular, the gradients associated with distinct
nodes of the same eigenfunction are automatically linearly independent. It also leads to a finite-dimensional
observation principle for mixed angular-momentum data under a natural transversality condition and for simultaneous
spectral--nodal data from a single eigenmode, where transversality is automatic. The principle provides exact
local feasible corrections, minimum-norm asymptotics, and, in the Hilbert setting, inverse-Gram formulas and local
uniqueness. For $p=2$ and $d\in{2,3}$, the required $C^2$ regularity of the nodal maps is established rather than assumed.

The second contribution is a global uniqueness theorem for inward displacements of the unique interior node of
the second radial mode. For $d\ge2$, $\ell=0$, a constant reference potential, and $p>(d+2)/2$, every such displacement
admits a unique global minimizer. The proof first selects the critical sign for every global minimizer and then reduces
the optimality system to a scalar weighted-mass balance between a focusing ball branch and a logistic annulus branch.
The strict opposite monotonicity of the two masses makes the balance parameter unique and consequently determines the
optimizer uniquely. This is an approach entirely different from local inverse-mapping theorem or
one-dimensional integrability arguments in \cite{HeWuXiaZhang2025}. Note that this result provides a genuinely global rigidity mechanism beyond the local inverse-mapping theory.

Throughout the paper we assume
\begin{equation}\label{eq:main-p-assumption}
 d\ge2,\qquad 1<p<\infty,\qquad p>\frac d2.
\end{equation}
All function spaces, duality pairings, potentials, and perturbations are
understood over the real field.  In particular, every admissible potential
$q\in L_d^p$ is real-valued, ensuring that the fixed-sector form is symmetric
and the associated realization is self-adjoint.

This paper is organized as follows. Section~\ref{sec:critical} develops the fixed-sector spectral and nodal
calculus, proves feasibility and existence, and derives the critical system.
Section~\ref{sec:reconstruction} establishes the local reconstruction theory,
the shooting representation, and the global rigidity theorem.  Section~\ref{sec:multi-node}
treats several nodes of one fixed mode, while Section~\ref{sec:mixed-data}
uses the abstract observation principle for mixed-sector and spectral--nodal
data.  The final section records the remaining open problems.

\section{The fixed-sector critical system}\label{sec:critical}

\subsection{Fixed-sector spectral preliminaries}
Let $Y_{\ell,k}$ be a real $L^2(\mathbb S^{d-1})$-normalized spherical
harmonic of degree $\ell$.  Let us define the profile form domain by the exact
sector identification
\begin{equation}\label{eq:sector-form-domain}
 \mathcal H_{\ell,0}:=
 \left\{u:(x\mapsto u(|x|)Y_{\ell,k}(x/|x|))\in H_0^1(B_R)\right\},
\end{equation}
which is a Hilbert space with norm
\begin{equation*}
 \|u\|_{\mathcal H_{\ell,0}}^2
 :=\int_0^R\left(|u'|^2+\frac{\gamma}{r^2}|u|^2+|u|^2\right)\dd\mu_d.
\end{equation*}
Equivalently, it is the completion, in the above norm,
of smooth profiles $u$ for which
$u(r)=r^\ell w(r^2)$ near $r=0$ and $u(R)=0$.  Thus the endpoint condition
at $0$ is precisely the Friedrichs condition, rather than an additional
pointwise boundary condition.  The definition is independent of the choice
of the normalized harmonic $Y_{\ell,k}$.

\begin{remark}
For a regular Sturm--Liouville equation, the operator domain is described by
ordinary traces at the endpoints, and perturbation theory may be based on
solutions initialized by Cauchy data there.  Here the quantity
$\gamma r^{-2}$ is singular for $\ell>0$ and the radial coefficient
$(d-1)r^{-1}$ is singular at the level of the differential expression.  The
condition at $0$ is therefore encoded by finite quadratic-form energy, not by
freely prescribing two traces.  The factorization $u=r^\ell V$ below is what
separates the Friedrichs branch from the non-admissible fundamental solution.
This distinction is essential later when differentiating eigenfunctions with
respect to $q$: the linearized solution must remain on the same Friedrichs
branch, otherwise the Lagrange identity may acquire an uncontrolled endpoint
term and the nodal derivative formula need not follow.
\end{remark}

For completeness, the equivalence follows directly from spherical-harmonic
projection.  If $U(x)=u(r)Y_{\ell,k}(\omega)$, then polar coordinates and
\(-\Delta_{\mathbb S^{d-1}}Y_{\ell,k}=\gamma Y_{\ell,k}\)
give the exact identity
\begin{equation}\label{eq:sector-isometry}
\begin{aligned}
 \int_{B_R}\bigl(|\nabla U|^2+|U|^2\bigr)\dd x=\int_0^R\left(|u'|^2+\frac{\gamma}{r^2}|u|^2\right)\dd\mu_d+\int_0^R|u|^2\dd\mu_d.
\end{aligned}
\end{equation}

For $q\in L_d^p$, define
\begin{equation}\label{eq:sector-form}
 \mathfrak h_q[u,v]
 :=\int_0^R\left(u'v'+\frac{\gamma}{r^2}uv+quv\right)\dd\mu_d,
 \qquad u,v\in\mathcal H_{\ell,0}.
\end{equation}
By the Sobolev and compact Sobolev embeddings on $B_R$, applied to
$u(r)Y_{\ell,k}(\omega)$, together with H\"older's inequality and
interpolation, the potential is infinitesimally form bounded with respect
to the free fixed-sector form.  More precisely, under
\eqref{eq:main-p-assumption}, for every $M>0$ and every $\epsilon>0$ there
exists $C_{M,\epsilon}>0$ such that
\begin{equation}\label{eq:form-bound}
 \left|\int_0^Rqu^2\dd\mu_d\right|
 \leq \epsilon\int_0^R\left(|u'|^2+\frac{\gamma}{r^2}|u|^2\right)\dd\mu_d
 +C_{M,\epsilon}\int_0^Ru^2\dd\mu_d
\end{equation}
for all $q\in L_d^p$ with $\|q\|_{L_d^p}\leq M$ and all
$u\in\mathcal H_{\ell,0}$.  The same Sobolev argument gives the compact
embedding
\begin{equation}\label{eq:sector-compact-embedding}
 \mathcal H_{\ell,0}\hookrightarrow L_d^{2p'}.
\end{equation}
Consequently, the form \eqref{eq:sector-form} is closed and bounded below
on $\mathcal H_{\ell,0}$.  Its Friedrichs operator in $L_d^2$ is denoted by
$H_{q,\ell}:=-u''-\frac{d-1}{r}u'
 +\frac{\ell(\ell+d-2)}{r^2}u+q(r)u$ and has compact resolvent.
The assumption $p>d/2$ also gives the endpoint integrability needed below:
\begin{equation}\label{eq:rq-integrable}
 \int_0^Rr|q(r)|\dd r<\infty.
\end{equation}
Indeed,
\[
 \int_0^Rr|q(r)|\dd r
 \leq\|q\|_{L_d^p}
 \left(\int_0^Rr^{\frac{p-d+1}{p-1}}\dd r\right)^{1/p'},
\]
and the last integral is finite exactly because $2p>d$.

\begin{lemma}
\label{lem:sector-volterra-endpoint}
Let $f,F\in L_d^p(0,R)$ and suppose that a Friedrichs
solution $u\in\mathcal H_{\ell,0}$ satisfies, near $r=0$,
\begin{equation}\label{eq:generic-sector-inhomogeneous}
 -u''-\frac{d-1}{r}u'+\frac{\gamma}{r^2}u+f(r)u
 =r^\ell F(r)
\end{equation}
in the weak sense.  Then $u(r)=r^\ell V(r)$, where $V$ has a continuous
extension to $r=0$.  If $a:=V(0)$, then
\begin{equation}\label{eq:sector-volterra-formula}
 V(r)=a+\int_0^r s^{1-(d+2\ell)}\int_0^s
 t^{d+2\ell-1}\big(f(t)V(t)-F(t)\big)\dd t\dd s.
\end{equation}
Moreover,
\begin{equation}\label{eq:sector-volterra-estimates}
 V(r)=a+O(r^{2-d/p}),\qquad
 V'(r)=O(r^{1-d/p}),\qquad rV'(r)\longrightarrow0
 \quad(r\downarrow0).
\end{equation}
For the homogeneous equation $F=0$, a nonzero Friedrichs solution has
$a\ne0$.
\end{lemma}
\begin{proof}
After writing $u=r^\ell V$, equation
\eqref{eq:generic-sector-inhomogeneous} becomes
\begin{equation}\label{eq:effective-dimensional-equation}
 -\big(r^{d+2\ell-1}V'\big)'+r^{d+2\ell-1}fV=r^{d+2\ell-1}F.
\end{equation}
Local one-dimensional regularity on every interval $[\delta,r_0]$ shows that
\[
 V\in AC_{\mathrm{loc}}(0,r_0],\qquad
 r^{d+2\ell-1}V'\in AC_{\mathrm{loc}}(0,r_0].
\]
Moreover, the right-hand side of
\begin{equation}\label{eq:effective-flux-equation}
 \big(r^{d+2\ell-1}V'(r)\big)'=r^{d+2\ell-1}\bigl(f(r)V(r)-F(r)\bigr)
\end{equation}
is integrable at the origin.  Indeed, since $u=r^\ell V$ and
$u\in\mathcal H_{\ell,0}\hookrightarrow L_d^{2p'}\hookrightarrow L_d^{p'}$ with $p'=\frac{p}{p-1}$,
\begin{align*}
 \int_0^{r_0}r^{d+2\ell-1}|fV|\dd r
 &=\int_0^{r_0}r^\ell |fu|\dd\mu_d
 \le r_0^\ell\|f\|_{L_d^p(0,r_0)}\|u\|_{L_d^{p'}(0,r_0)},\\
 \int_0^{r_0}r^{d+2\ell-1}|F|\dd r
 &\le r_0^{2\ell}\mu_d(0,r_0)^{1/p'}
       \|F\|_{L_d^p(0,r_0)}.
\end{align*}
Consequently the finite limit
\(C:=\lim_{r\downarrow0}r^{d+2\ell-1}V'(r)\)
exists.  If $C\ne0$, then
$V'(r)=Cr^{1-(d+2\ell)}+o(r^{1-(d+2\ell)})$.  Hence
\[
 V(r)=\begin{cases}
 C\log r+o(|\log r|),&d+2\ell=2,\\[1mm]
 \dfrac{C}{2-(d+2\ell)}r^{2-(d+2\ell)}+o(r^{2-(d+2\ell)}),&d+2\ell>2,
 \end{cases}
 \qquad r\downarrow0.
\]
The corresponding profile $u=r^\ell V$ has infinite form energy near zero,
contradicting $u\in\mathcal H_{\ell,0}$.  Thus $C=0$, and
\eqref{eq:effective-flux-equation} gives
\begin{equation}\label{eq:once-integrated-effective-equation}
 V'(r)=r^{1-(d+2\ell)}\int_0^r t^{d+2\ell-1}\bigl(f(t)V(t)-F(t)\bigr)\dd t,
\end{equation}
which implies \eqref{eq:sector-volterra-formula}.
We next show that the zero-flux solution is the regular Volterra branch.
For each $a\in\mathbb R$, the right-hand side of
\eqref{eq:sector-volterra-formula} defines a contraction on a bounded ball of
$C([0,r_0])$ when $r_0$ is sufficiently small.  It therefore produces a
unique bounded solution $V_a$ satisfying $V_a(0)=a$.  In the homogeneous
case, the regular solution $\phi$ with $\phi(0)=1$ is nonzero on a smaller
interval.  Reduction of order then gives a second solution
\[
 \theta(r)=\phi(r)\int_r^{r_0}\frac{s^{1-(d+2\ell)}}{\phi(s)^2}\dd s,
\]
for which $\lim_{r\downarrow0}r^{d+2\ell-1}\theta'(r)\ne0$.  Hence the kernel of
the flux functional among homogeneous solutions is precisely
$\operatorname{span}\{\phi\}$.  Given the original inhomogeneous solution
$V$, choose $a$ so that $V_a(r_0)=V(r_0)$; this is possible because the map
$a\mapsto V_a(r_0)$ has nonzero linear part $\phi(r_0)$.  The difference
$V-V_a$ is homogeneous, has zero endpoint flux, and vanishes at $r_0$;
therefore it is identically zero.  Thus $V$ is bounded, has the finite limit
$a=V(0)$, and integrating \eqref{eq:once-integrated-effective-equation}
proves \eqref{eq:sector-volterra-formula}.  This regular--singular alternative
is consistent with the fundamental solutions constructed in
\cite[Lemmas~3.2 and~3.8]{KostenkoTeschl2011}.

It remains to record the quantitative estimates.  H\"older's inequality gives
\begin{equation}\label{eq:effective-holder-estimate}
 \int_0^r t^{d+2\ell-1}|f(t)|\dd t
 \le C r^{d+2\ell-d/p}\|f\|_{L_d^p(0,r)},
\end{equation}
and the same estimate holds with $F$ in place of $f$.  Since $V$ is bounded
near zero, \eqref{eq:once-integrated-effective-equation} yields
\[
 |V'(r)|\le C r^{1-d/p}
 \left(\|f\|_{L_d^p(0,r)}\|V\|_{L^\infty(0,r)}
       +\|F\|_{L_d^p(0,r)}\right).
\]
Because $2-d/p>0$, integration proves
\eqref{eq:sector-volterra-estimates}.

Finally, if $F=0$ and $a=0$, the Volterra formula gives
\[
 \|V\|_{L^\infty(0,r)}
 \le C r^{2-d/p}\|f\|_{L_d^p(0,r)}
       \|V\|_{L^\infty(0,r)}.
\]
For sufficiently small $r$ the coefficient on the right is less than one,
so $V=0$ near $0$; uniqueness for the regular equation then gives
$u\equiv0$.  Thus every nonzero homogeneous Friedrichs solution has
$a\ne0$.
\end{proof}

\begin{remark}
The zero-flux condition obtained in the Volterra argument is not an
additional boundary condition.  By the characterization of space
$\mathcal H_{\ell,0}$,
the regular zero-flux solution is precisely the branch belonging to the
Friedrichs form domain.  The second fundamental solution has non-integrable
quadratic-form energy near $r=0$.  Hence the Volterra formula parametrizes
exactly the Friedrichs branch.
\end{remark}

\begin{proposition}\label{prop:sector-spectral-oscillation}
The spectrum of $H_{q,\ell}$ consists of simple eigenvalues
\begin{equation}\label{eq:sector-spectrum}
 \lambda_1(q)<\lambda_2(q)<\cdots,\qquad
 \lambda_m(q)\longrightarrow+\infty.
\end{equation}
An eigenfunction $E_m$ has a continuous representative on $[0,R]$, belongs
to $C^1([a,R])$ for every $a>0$, and satisfies
\begin{equation}\label{eq:sector-endpoint-asymptotics}
 r^{-\ell}E_m(r)\longrightarrow a_m,\qquad
 r^{1-\ell}E_m'(r)\longrightarrow \ell a_m
 \quad(r\downarrow0)
\end{equation}
for some $a_m\neq0$.  The eigenfunction associated with $\lambda_m(q)$ has
exactly $m-1$ simple zeros in $(0,R)$.
\end{proposition}

\begin{proof}
Since the quadratic form is closed and lower semibounded, and its
form domain is compactly embedded into the underlying weighted
\(L^2\)-space, the associated self-adjoint operator has compact
resolvent.
Let \(E\) be a nontrivial eigenfunction corresponding to the angular
momentum \(\ell\), and set
\[
v(r):=r^{(d-1)/2}E(r).
\]
Then the radial eigenvalue equation is transformed into
\[
-v''(r)
+\left(
\frac{\nu^2-\frac14}{r^2}+q(r)
\right)v(r)
=\lambda v(r).
\]
To determine the behavior at the singular endpoint, write \(E(r)=r^\ell V(r)\).
A direct substitution into the radial equation gives
\[
V''(r)+\frac{d+2\ell-1}{r}V'(r)
=(q(r)-\lambda)V(r).
\]
Note that Lemma~\ref{lem:sector-volterra-endpoint} applies with \(f=q-\lambda\) and \(F=0\). Consequently,
there exists \(a_m\in\mathbb R\) such that
\[
V(r)\longrightarrow a_m,
\qquad
rV'(r)\longrightarrow0
\quad\text{as }r\downarrow0.
\]
Moreover, \(a_m\neq0\). Indeed, if \(a_m=0\), the endpoint uniqueness
statement in Lemma~\ref{lem:sector-volterra-endpoint} would imply \(V\equiv0\), and hence
\(E\equiv0\), contradicting the nontriviality of the eigenfunction.
It follows that
\[
r^{-\ell}E(r)=V(r)\longrightarrow a_m,
\]
and, since
\[
E'(r)=\ell r^{\ell-1}V(r)+r^\ell V'(r),
\]
we also have
\[
r^{1-\ell}E'(r)
=\ell V(r)+rV'(r)
\longrightarrow \ell a_m.
\]
This proves \eqref{eq:sector-endpoint-asymptotics} and yields the continuous extension of
\(E\) to the singular endpoint. Finally, standard local regularity
for the radial equation away from \(r=0\) gives
\[
E\in C^1([a,R])
\qquad\text{for every }a\in(0,R).
\]
If two eigenfunctions correspond to the same eigenvalue, their weighted
Wronskian
\[
 r^{d-1}(E_1E_2'-E_1'E_2)
\]
is constant and vanishes at $R$; local uniqueness on every interval
$[a,R]$ implies linear dependence.  Interior zeros are simple by the same
local uniqueness.  Finally, the singular oscillation theorem for separated self-adjoint
Sturm--Liouville problems applies: the coefficients are real, the right
endpoint is regular with the Dirichlet condition, the left endpoint carries
the Friedrichs condition, and \eqref{eq:rq-integrable} supplies the required
Bessel-endpoint integrability.  The oscillation theory and the corresponding
spectral summary are given in \cite[Ch.~6, pp.~131--140; Ch.~10,
pp.~208--211]{Zettl}.  It follows that the eigenfunction associated with the
$m$th eigenvalue has exactly $m-1$ interior zeros.
\end{proof}

\subsubsection{Weak convergence of fixed-sector eigenpairs}

The following compactness observation is useful because weak convergence of the potentials is upgraded, through the compact Sobolev embedding, to uniform convergence of the corresponding multiplication forms on bounded subsets of the energy space.

\begin{lemma}\label{lem:uniform-form-convergence}
Assume \eqref{eq:main-p-assumption} and $q_n\weakto q$ in $L_d^p$.  Then
\begin{equation}\label{eq:uniform-form-convergence}
 \sup_{\substack{\|u\|_{ \mathcal H_{\ell,0}}\le1\\ \|v\|_{ \mathcal H_{\ell,0}}\le1}}
 \left|\int_0^R (q_n-q)u v\dd\mu_d\right|\longrightarrow0.
\end{equation}
\end{lemma}

\begin{proof}
By the compact embedding \eqref{eq:sector-compact-embedding}, $ \mathcal H_{\ell,0}\hookrightarrow L_d^{2p'}$.  Hence the set
\[
 \mathcal K:=\{uv:\ \|u\|_{ \mathcal H_{\ell,0}}\le1,\ \|v\|_{ \mathcal H_{\ell,0}}\le1\}
\]
is relatively compact in $L_d^{p'}$: indeed, the product map
$L_d^{2p'}\times L_d^{2p'}\to L_d^{p'}$ is continuous, and the image of the product of two relatively compact sets is relatively compact.  The functionals
\[
 \ell_n(f):=\int_0^R(q_n-q)f\dd\mu_d
\]
are uniformly bounded on $L_d^{p'}$ and converge pointwise to zero.
 A uniformly bounded sequence of linear functionals which converges pointwise on a Banach space converges
 uniformly on every compact subset.  Applying this to the closure of $\mathcal K$ proves \eqref{eq:uniform-form-convergence}.
\end{proof}

We next record the compactness properties of fixed-sector eigenpairs used below.
\begin{proposition}\label{prop:eigenpair-compactness}
Assume \eqref{eq:main-p-assumption} and let $q_n\weakto q$ in $L_d^p$.  Fix $m\ge1$ and choose normalized eigenfunctions $E_{m,n}:=E_m(\cdot;q_n)$ so that, after discarding finitely many terms if necessary, $\langle E_{m,n},E_m(\cdot;q)\rangle_{L_d^2}\ge0$.  Then
\begin{align}
 \lambda_m(q_n)&\longrightarrow\lambda_m(q),\\
 E_{m,n}&\longrightarrow E_m(\cdot;q)\quad\text{strongly in } \mathcal H_{\ell,0},\label{eq:H1-conv}
\end{align}
and, for every $a\in(0,R)$,
\begin{equation}\label{eq:C1-conv}
 E_{m,n}\longrightarrow E_m(\cdot;q)
 \quad\text{in }C^1([a,R]).
\end{equation}
\end{proposition}

\begin{proof}
By Lemma~\ref{lem:uniform-form-convergence}, the closed forms $\mathfrak h_{q_n}$ converge to $\mathfrak h_q$ in form norm on the common domain $ \mathcal H_{\ell,0}$.  The uniform lower bound furnished by \eqref{eq:form-bound} and the min--max principle therefore give
\[
 \lambda_j(q_n)\longrightarrow\lambda_j(q),\qquad 1\le j\le m.
\]
The eigenvalue identity and the uniform form bound show that $E_{m,n}$ is bounded in $ \mathcal H_{\ell,0}$.  Along any subsequence,
\[
 E_{m,n_k}\weakto E_*\text{ in } \mathcal H_{\ell,0},
 \qquad
 E_{m,n_k}\to E_*\text{ in }L_d^{2p'}\cap L_d^2.
\]
The strong $L_d^2$ convergence preserves the normalization.  For every $\phi\in \mathcal H_{\ell,0}$,
\begin{align*}
 &\int_0^R q_{n_k}E_{m,n_k}\phi\dd\mu_d
 -\int_0^R qE_*\phi\dd\mu_d\\
 &\quad=
 \int_0^R q_{n_k}(E_{m,n_k}-E_*)\phi\dd\mu_d
 +\int_0^R(q_{n_k}-q)E_*\phi\dd\mu_d\longrightarrow0.
\end{align*}
Thus $E_*$ is a normalized eigenfunction of $H_{q,\ell}$ associated with $\lambda_m(q)$.  Fixed-sector simplicity implies $E_*=\pm E_m(\cdot;q)$.  After the prescribed sign choice, every subsequence has the same limit.
Moreover,
\begin{align*}
 \int_0^R q_nE_{m,n}^2\dd\mu_d-
 \int_0^R qE_m^2\dd\mu_d
 &=\int_0^R q_n(E_{m,n}^2-E_m^2)\dd\mu_d\\
 &\quad+\int_0^R(q_n-q)E_m^2\dd\mu_d\longrightarrow0.
\end{align*}
Together with the convergence of the eigenvalues and the normalization, this yields convergence of the fixed-sector form energies.  Weak convergence in $ \mathcal H_{\ell,0}$ plus convergence of the $ \mathcal H_{\ell,0}$ norms proves \eqref{eq:H1-conv}.

Fix $a>0$ and put $s=\min\{p,2\}>1$.  On $[a,R]$ the fixed-sector equation is nondegenerate.  The strong $H^1(a,R)$ convergence gives a uniform $L^\infty(a,R)$ bound, while
\[
 E_{m,n}''=-\frac{d-1}{r}E_{m,n}'+\left(\frac{\gamma}{r^2}+q_n-\lambda_m(q_n)\right)E_{m,n}
\]
shows that $E_{m,n}$ is bounded in $W^{2,s}(a,R)$.  Since
$W^{2,s}(a,R)\hookrightarrow C^1([a,R])$ compactly, every subsequence has a further subsequence converging in $C^1([a,R])$ to $E_m(\cdot;q)$.  Uniqueness of the limit proves \eqref{eq:C1-conv} for the full sequence.
\end{proof}

\subsection{Fixed-sector nodes as nonlinear functionals}

Let $T=T_{i,m}(q)$ and normalize $E=E_m(\cdot;q)$ by
\begin{equation}\label{eq:normalization}
 \int_0^R E(r)^2\dd\mu_d(r)=1.
\end{equation}
Set
\begin{align}\label{eq:A-B-Gamma}
 A=A_{i,m}(q)&:=\int_0^T E^2\dd\mu_d,
 &B=B_{i,m}(q)&:=\int_T^R E^2\dd\mu_d=1-A,\\
 \Gamma=\Gamma_{i,m}(q)&:=T^{d-1}E'(T)^2.\notag
\end{align}
Both $A$ and $B$ are positive, and $\Gamma>0$ because every interior fixed-sector node is simple.

\begin{proposition}\label{prop:eigenpair-differentiability}
Fix $q\in L_d^p$ and a normalized fixed-sector eigenpair
$(\lambda,E)=(\lambda_m(q),E_m(\cdot;q))$.  There exists a neighborhood $\mathcal U$ of $q$ in $L_d^p$ and a locally sign-fixed choice
\[
 \widetilde q\longmapsto
 (\lambda_m(\widetilde q),E_m(\cdot;\widetilde q))
 \in\R\times \mathcal H_{\ell,0}
\]
which is continuously Fr\'echet differentiable.  For every $a\in(0,R)$ and $s=\min\{p,2\}>1$, the same map is continuously Fr\'echet differentiable with values in $\R\times W^{2,s}(a,R)$, and hence with values in $\R\times C^1([a,R])$.
For $h\in L_d^p$, write
\begin{equation}\label{eq:Zdef}
 \dot\lambda=D\lambda_m(q)[h],\qquad Z=DE_m(q)[h].
\end{equation}
Then
\begin{equation}\label{eq:Hellmann-Feynman}
 \dot\lambda=\int_0^R h(r)E(r)^2\dd\mu_d(r),
\end{equation}
and $Z$ is the unique solution orthogonal to $E$ of
\begin{equation}\label{eq:linearized-eigenfunction}
 \begin{cases}
 -\big(r^{d-1}Z'\big)'+\gamma r^{d-3}Z+r^{d-1}(q-\lambda)Z
 =r^{d-1}(\dot\lambda-h)E,\\
 Z\text{ satisfies the Friedrichs condition at }0,\qquad Z(R)=0,\\
 \displaystyle \int_0^R ZE\dd\mu_d=0.
 \end{cases}
\end{equation}
\end{proposition}

\begin{proof}
Let $J:\mathcal H_{\ell,0}\to(\mathcal H_{\ell,0})^*$ denote the natural
$L_d^2$ embedding,
\[
 \langle Ju,\phi\rangle=\langle u,\phi\rangle_{L_d^2},
 \qquad u,\phi\in\mathcal H_{\ell,0}.
\]
For $\widetilde q\in L_d^p$, let $M_{\widetilde q}:\mathcal H_{\ell,0}
\to(\mathcal H_{\ell,0})^*$ be the multiplication-form operator
\[
 \langle M_{\widetilde q}u,\phi\rangle
 :=\int_0^R\widetilde q\,u\phi\dd\mu_d.
\]
Thus, if $A_0$ denotes the free fixed-sector form operator, then
$A_{\widetilde q}=A_0+M_{\widetilde q}$ and
\(\langle A_{\widetilde q}u,\phi\rangle
 =\mathfrak h_{\widetilde q}[u,\phi].\)
By H\"older's inequality and
$\mathcal H_{\ell,0}\hookrightarrow L_d^{2p'}$,
\[
 \|M_{\widetilde q}u\|_{(\mathcal H_{\ell,0})^*}
 \le C\|\widetilde q\|_{L_d^p}\|u\|_{\mathcal H_{\ell,0}}.
\]
Hence $(\widetilde q,u)\mapsto M_{\widetilde q}u$ is a continuous
bilinear map. Define
\[
 \mathscr G(\widetilde q,u,\eta)
 :=\left(A_{\widetilde q}u-\eta Ju,
 \frac12\big(\|u\|_{L_d^2}^2-1\big)\right).
\]
Then
\[
 \mathscr G:L_d^p\times\mathcal H_{\ell,0}\times\R
 \longrightarrow(\mathcal H_{\ell,0})^*\times\R
\]
is continuously Fr\'echet differentiable.  At the normalized eigenpair
$(q,E,\lambda)$, its derivative with respect to $(u,\eta)$ is
\[
 \mathcal L(Z,\xi)
 :=D_{(u,\eta)}\mathscr G(q,E,\lambda)[Z,\xi]
 =\left((A_q-\lambda J)Z-\xi JE,
 \langle Z,E\rangle_{L_d^2}\right).
\]
We prove that $\mathcal L$ is an isomorphism. Set
\[
 X:=\{Z\in\mathcal H_{\ell,0}:\langle Z,E\rangle_{L_d^2}=0\},
 \qquad
 Y:=\{F\in(\mathcal H_{\ell,0})^*:F(E)=0\}.
\]
We first claim that
\begin{equation}\label{eq:reduced-isomorphism}
 A_q-\lambda J:X\longrightarrow Y
\end{equation}
is a bounded isomorphism.  By the uniform form bound, one may choose
$c>0$ so large that
$A_q+cJ:\mathcal H_{\ell,0}\to(\mathcal H_{\ell,0})^*$ is coercive;
hence it is an isomorphism by the Lax--Milgram theorem.  The map $J$ is
compact, since it factors through the compact embedding
$\mathcal H_{\ell,0}\hookrightarrow L_d^2$.  Therefore
\[
 A_q-\lambda J=(A_q+cJ)-(\lambda+c)J
\]
is Fredholm of index zero.  Since $\lambda$ is a simple fixed-sector
eigenvalue,
\[
 \ker(A_q-\lambda J)=\operatorname{span}\{E\}.
\]
Moreover, symmetry gives
\[
 \langle(A_q-\lambda J)Z,E\rangle
 =\langle Z,(A_q-\lambda J)E\rangle=0
\]
for all $Z\in \mathcal H_{\ell,0}$,
so $\operatorname{Ran}(A_q-\lambda J)\subset Y$.  The Fredholm index is
zero and the kernel is one-dimensional, hence the range has codimension
one.  Since $Y$ also has codimension one, it follows that
\[
 \operatorname{Ran}(A_q-\lambda J)=Y.
\]
Restricting the domain to $X$ removes the kernel; surjectivity onto $Y$
follows by subtracting the $E$-component of any preimage.  This proves
\eqref{eq:reduced-isomorphism}.  The bounded inverse theorem therefore gives
\begin{equation}\label{eq:reduced-resolvent-bound}
 \|Z\|_{\mathcal H_{\ell,0}}
 \le C\|(A_q-\lambda J)Z\|_{(\mathcal H_{\ell,0})^*},
 \qquad Z\in X.
\end{equation}
To prove invertibility of the full linearization, let
$(F,a)\in(\mathcal H_{\ell,0})^*\times\R$ be given and solve
\[
 (A_q-\lambda J)Z-\xi JE=F,
 \qquad
 \langle Z,E\rangle_{L_d^2}=a.
\]
Since $\|E\|_{L_d^2}=1$, write $Z=aE+Z_\perp$ with $Z_\perp\in X$.
Pairing the first equation with $E$ yields
\(-\xi=F(E)\),
and therefore $F+\xi JE\in Y$.  By \eqref{eq:reduced-isomorphism},
\[
 Z_\perp
 =\big((A_q-\lambda J)|_X\big)^{-1}(F+\xi JE)
\]
is uniquely determined.  Furthermore,
\[
 |\xi|\le C\|F\|_{(\mathcal H_{\ell,0})^*},
 \qquad
 \|Z_\perp\|_{\mathcal H_{\ell,0}}
 \le C\|F\|_{(\mathcal H_{\ell,0})^*},
\]
so
\[
 \|Z\|_{\mathcal H_{\ell,0}}+|\xi|
 \le C\big(\|F\|_{(\mathcal H_{\ell,0})^*}+|a|\big).
\]
Thus $\mathcal L$ is a bounded isomorphism.
The Banach-space implicit function theorem now gives neighborhoods of $q$
and $(E,\lambda)$ and a unique continuously Fr\'echet differentiable local
branch
\[
 \widetilde q\longmapsto (u(\widetilde q),\eta(\widetilde q))
\]
of normalized eigenpairs satisfying
$\mathscr G(\widetilde q,u(\widetilde q),\eta(\widetilde q))=0$ and
$(u(q),\eta(q))=(E,\lambda)$.  By
Proposition~\ref{prop:eigenpair-compactness}, the locally sign-fixed
$m$th normalized eigenpair converges to $(E,\lambda)$ in
$\mathcal H_{\ell,0}\times\R$ as $\widetilde q\to q$ in $L_d^p$.
After shrinking the neighborhood if necessary and choosing the sign by
$\langle E_m(\cdot;\widetilde q),E\rangle_{L_d^2}>0$, this eigenpair lies
in the implicit-function neighborhood.  The local uniqueness in the
implicit function theorem therefore identifies the branch with
\[
 (u(\widetilde q),\eta(\widetilde q))
 =(E_m(\cdot;\widetilde q),\lambda_m(\widetilde q)).
\]
This proves the asserted $C^1$ dependence with values in
$\mathcal H_{\ell,0}\times\R$.
Let $h\in L_d^p$ and set
$\dot\lambda=D\lambda_m(q)[h]$ and $Z=DE_m(q)[h]$.  Differentiating the
identity $\mathscr G(q,E,\lambda)=0$ in the direction $h$ gives
\begin{equation}\label{eq:linearized-form-equation}
 (A_q-\lambda J)Z=(\dot\lambda J-M_h)E,
 \qquad
 \langle Z,E\rangle_{L_d^2}=0.
\end{equation}
Pairing the first equation with $E$ and using symmetry and
$(A_q-\lambda J)E=0$ yields
\[
 0=\dot\lambda\|E\|_{L_d^2}^2
   -\int_0^R hE^2\dd\mu_d,
\]
which proves \eqref{eq:Hellmann-Feynman}.  Equation
\eqref{eq:linearized-form-equation}, written in radial weak form, is exactly
\eqref{eq:linearized-eigenfunction}.  Since $Z\in\mathcal H_{\ell,0}$,
the Friedrichs condition at $0$ and the Dirichlet condition at $R$ are
already encoded in the form domain.  Finally, the right-hand side of
\eqref{eq:linearized-form-equation} belongs to $Y$ by
\eqref{eq:Hellmann-Feynman}, so \eqref{eq:reduced-isomorphism} gives the
claimed uniqueness of the solution orthogonal to $E$.

It remains to prove differentiability in the stronger local topology.  Fix
$a\in(0,R)$, put $b=a/2$, and let $s=\min\{p,2\}>1$.  On $[b,R]$ the
weighted and unweighted Lebesgue norms are equivalent, and
$\mathcal H_{\ell,0}$ controls the usual $H^1(b,R)$ norm.  In particular,
$H^1(b,R)\hookrightarrow L^\infty(b,R)$.  From the boundedness of the
Fr\'echet derivative in $\mathcal H_{\ell,0}\times\R$,
\[
 \|Z\|_{\mathcal H_{\ell,0}}+|\dot\lambda|
 \le C\|h\|_{L_d^p}.
\]
On $[b,R]$, \eqref{eq:linearized-eigenfunction} gives, in the
distributional sense,
\[
 Z''=-\frac{d-1}{r}Z'
 +\left(\frac{\gamma}{r^2}+q-\lambda\right)Z
 -(\dot\lambda-h)E.
\]
Here $Z,E\in H^1(b,R)\subset L^\infty(b,R)$, $q,h\in L^p(b,R)$,
and $s\le p,2$.  Hence every term on the right belongs to $L^s(b,R)$ and
\begin{equation}\label{eq:local-derivative-bound}
 \|Z\|_{W^{2,s}(a,R)}\le C_a\|h\|_{L_d^p}.
\end{equation}
Thus the derivative takes values in $W^{2,s}(a,R)$.

We next prove Fr\'echet differentiability in this topology.  For
$h\to0$ in $L_d^p$, write
\[
 E_h=E_m(\cdot;q+h),\qquad
 \lambda_h=\lambda_m(q+h),\qquad
 R_h=E_h-E-Z,
\]
and
\[
 \rho_h=\lambda_h-\lambda-\dot\lambda.
\]
The already established differentiability in
$\mathcal H_{\ell,0}\times\R$ gives
\begin{equation}\label{eq:H-remainder-small}
 \|R_h\|_{\mathcal H_{\ell,0}}=o(\|h\|_{L_d^p}),
 \qquad
 |\rho_h|=o(\|h\|_{L_d^p}).
\end{equation}
Subtracting the equations for $E_h$, $E$, and $Z$ gives
\begin{equation}\label{eq:eigenpair-remainder}
 (A_q-\lambda J)R_h
 =\rho_hJE+(\dot\lambda+\rho_h)J(E_h-E)-M_h(E_h-E).
\end{equation}
Equivalently, on $[b,R]$,
\[
 -R_h''-\frac{d-1}{r}R_h'
 +\left(\frac{\gamma}{r^2}+q-\lambda\right)R_h
 =F_h,
\]
where
\[
 F_h:=\rho_hE+(\dot\lambda+\rho_h-h)(E_h-E).
\]
Since the $\mathcal H_{\ell,0}$-valued eigenpair map is differentiable,
\[
 \|E_h-E\|_{H^1(b,R)}=O(\|h\|_{L_d^p}),
 \qquad
 \|R_h\|_{H^1(b,R)}=o(\|h\|_{L_d^p}).
\]
The one-dimensional Sobolev embedding therefore gives
\[
 \|E_h-E\|_{L^\infty(b,R)}=O(\|h\|_{L_d^p}),
 \qquad
 \|R_h\|_{L^\infty(b,R)}=o(\|h\|_{L_d^p}).
\]
Together with $|\dot\lambda|=O(\|h\|_{L_d^p})$ and
\eqref{eq:H-remainder-small}, this implies \(\|F_h\|_{L^s(b,R)}=o(\|h\|_{L_d^p})\).
Moreover,
\[
 \|R_h'\|_{L^s(b,R)}+\|R_h\|_{L^s(b,R)}
 =o(\|h\|_{L_d^p}),
\]
and, since $q\in L^p(b,R)$ and $R_h\in L^\infty(b,R)$, \(\|qR_h\|_{L^s(b,R)}=o(\|h\|_{L_d^p})\).
The differential equation therefore yields directly
\[
 \|R_h''\|_{L^s(b,R)}=o(\|h\|_{L_d^p}),
\]
so
\begin{equation}\label{eq:local-W2s-remainder}
 \|R_h\|_{W^{2,s}(a,R)}=o(\|h\|_{L_d^p}).
\end{equation}
This proves Fr\'echet differentiability of the eigenfunction map with values
in $W^{2,s}(a,R)$.

It remains only to prove continuity of this derivative.  Let $q_n\to q$ in
$L_d^p$, choose the same local sign convention, and write
\[
 E_n=E_m(\cdot;q_n),\qquad \lambda_n=\lambda_m(q_n),
\]
and, for $\|h\|_{L_d^p}\le1$,
\[
 Z_n=DE_m(q_n)[h],\qquad Z=DE_m(q)[h],
\]
\[
 \dot\lambda_n=\int_0^RhE_n^2\dd\mu_d,
 \qquad
 \dot\lambda=\int_0^RhE^2\dd\mu_d.
\]
The $C^1$ dependence already obtained by the implicit function theorem in $\mathcal H_{\ell,0}$ gives
\begin{equation}\label{eq:H-derivative-continuity}
 \sup_{\|h\|_{L_d^p}\le1}
 \|Z_n-Z\|_{\mathcal H_{\ell,0}}\longrightarrow0.
\end{equation}
Also, by \eqref{eq:Hellmann-Feynman} we have
\[
 \sup_{\|h\|_{L_d^p}\le1}|\dot\lambda_n-\dot\lambda|
 \le\|E_n^2-E^2\|_{L_d^{p'}}\longrightarrow0,
\]
so the eigenvalue derivative is continuous in operator norm.  On $[b,R]$
the difference $W_n:=Z_n-Z$ satisfies
\begin{align}
 &(A_q-\lambda J)W_n\notag\\
 &\quad=(\dot\lambda_n-\dot\lambda)JE_n
 +(\dot\lambda J-M_h)(E_n-E)
 -M_{q_n-q}Z_n+(\lambda_n-\lambda)JZ_n.
 \label{eq:linearized-difference-equation}
\end{align}
Because $E_n\to E$ in $\mathcal H_{\ell,0}$, the functions $E_n$ are
uniformly bounded in $L^\infty(b,R)$.  Moreover,
\eqref{eq:H-derivative-continuity} implies that $Z_n$ is uniformly bounded
in $H^1(b,R)$, uniformly for $\|h\|_{L_d^p}\le1$.  Hence the right-hand
side of \eqref{eq:linearized-difference-equation}, viewed as a function on
$[b,R]$, tends to zero in $L^s(b,R)$ uniformly for
$\|h\|_{L_d^p}\le1$.  From \eqref{eq:H-derivative-continuity},
\[
 \sup_{\|h\|_{L_d^p}\le1}
 \big(\|W_n'\|_{L^s(b,R)}+\|W_n\|_{L^\infty(b,R)}\big)
 \longrightarrow0.
\]
Since $q\in L^p(b,R)$, the differential equation for $W_n$ then gives
\[
 \sup_{\|h\|_{L_d^p}\le1}
 \|W_n''\|_{L^s(b,R)}\longrightarrow0.
\]
Consequently,
\[
 \sup_{\|h\|_{L_d^p}\le1}
 \|DE_m(q_n)[h]-DE_m(q)[h]\|_{W^{2,s}(a,R)}\longrightarrow0.
\]
Thus the derivative is continuous as a map into
$\mathcal L(L_d^p,W^{2,s}(a,R))$.  Since the same argument applies at every
point of the implicit-function neighborhood, the eigenpair map is
continuously Fr\'echet differentiable with values in
$\R\times W^{2,s}(a,R)$.  Finally,
$W^{2,s}(a,R)\hookrightarrow C^1([a,R])$ for $s>1$, which gives the last
assertion.
\end{proof}

\begin{lemma}\label{lem:endpoint-regularity}
Let $E$ and $Z$ be as in Proposition~\ref{prop:eigenpair-differentiability}.
Then, for suitable constants $a_E$ and $a_Z$,
\begin{align}
 E(r)&=a_Er^\ell+o(r^\ell),&
 E'(r)&=\ell a_Er^{\ell-1}+o(r^{\ell-1}),\label{eq:E-sector-asymptotics}\\
 Z(r)&=a_Zr^\ell+o(r^\ell),&
 Z'(r)&=\ell a_Zr^{\ell-1}+o(r^{\ell-1})\label{eq:Z-sector-asymptotics}
\end{align}
as $r\downarrow0$, with the derivative statements interpreted as
$rE'(r)\to0$ and $rZ'(r)\to0$ when $\ell=0$.  Consequently,
\begin{equation}\label{eq:Wronskian-origin}
 \lim_{r\downarrow0}
 r^{d-1}\big(Z(r)E'(r)-E(r)Z'(r)\big)=0.
\end{equation}
\end{lemma}

\begin{proof}
Write $E=r^\ell V$.  Lemma~\ref{lem:sector-volterra-endpoint}, applied with
$f=q-\lambda$ and $F=0$, gives
$V(r)\to a_E$ and $rV'(r)\to0$, which is exactly
\eqref{eq:E-sector-asymptotics}.
Next write $Z=r^\ell Y$.  Dividing
\eqref{eq:linearized-eigenfunction} by $r^\ell$ gives
\[
 -\big(r^{d+2\ell-1}Y'\big)'
 +r^{d+2\ell-1}(q-\lambda)Y
 =r^{d+2\ell-1}(\dot\lambda-h)V.
\]
The function $V$ is bounded near zero and
$(\dot\lambda-h)V\in L_d^p$; hence
Lemma~\ref{lem:sector-volterra-endpoint} applies with
$F=(\dot\lambda-h)V$.  It gives $Y(r)\to a_Z$ and $rY'(r)\to0$, proving
\eqref{eq:Z-sector-asymptotics}.
Finally, direct substitution yields
\[
 r^{d-1}(ZE'-EZ')
 =r^{d+2\ell-1}(YV'-VY').
\]
The boundedness of $V,Y$ and the limits $rV'(r),rY'(r)\to0$ show that the
right-hand side is $o(r^{d+2\ell-2})$, and therefore tends to zero because
$d+2\ell-2\ge0$; in the borderline case $d=2$, $\ell=0$, it is still
$o(1)$.  This proves \eqref{eq:Wronskian-origin}.
\end{proof}

\begin{lemma}\label{lem:node-derivative}
Let $T=T_{i,m}(q)$.  Then
\begin{equation}\label{eq:Tdot-first}
 D T_{i,m}(q)[h]
 =\frac{1}{T^{d-1}E'(T)^2}
 \left(
 \int_0^T hE^2\dd\mu_d
 -A\int_0^R hE^2\dd\mu_d
 \right).
\end{equation}
Equivalently,
\begin{equation}\label{eq:node-derivative-main}
 D T_{i,m}(q)[h]
 =\int_0^R g_{i,m}(r;q)h(r)\dd\mu_d(r),
\end{equation}
where
\begin{equation*}
 g_{i,m}(r;q)
 =\frac{B}{\Gamma}E(r)^2\ind_{(0,T)}(r)
 -\frac{A}{\Gamma}E(r)^2\ind_{(T,R)}(r)
\end{equation*}
with $\Gamma$ defined in \eqref{eq:A-B-Gamma}.
\end{lemma}

\begin{proof}
By Proposition~\ref{prop:eigenpair-differentiability}, the map
$(r,\widetilde q)\mapsto E_m(r;\widetilde q)$ is $C^1$ in a neighborhood of $(T,q)$.  Since $E(T;q)=0$ and $E'(T;q)\ne0$, the scalar implicit function theorem gives
\begin{equation}\label{eq:Tdot-implicit}
 D T_{i,m}(q)[h]=-\frac{Z(T)}{E'(T)},
\end{equation}
where $Z$ was defined in \eqref{eq:Zdef}.
Multiply \eqref{eq:linearized-eigenfunction} by $E$, subtract the eigenvalue equation multiplied by $Z$, and use the Lagrange identity.  One obtains

\[
 \left[r^{d-1}(ZE'-EZ')\right]'
 =r^{d-1}(\dot\lambda-h)E^2.
\]
Lemma~\ref{lem:endpoint-regularity} shows that the singular endpoint
contributes no boundary term to this identity.  Thus the calculation below
is valid on the complete Friedrichs sector, not merely on intervals bounded
away from the origin.
Integrating from $0$ to $T$ and using \eqref{eq:Wronskian-origin}, we obtain,
since $E(T)=0$,
\[
 T^{d-1}Z(T)E'(T)
 =\dot\lambda A-\int_0^T hE^2\dd\mu_d.
\]
Together with \eqref{eq:Tdot-implicit} and \eqref{eq:Hellmann-Feynman}, this proves \eqref{eq:Tdot-first}.
Splitting the total integral at $T$ gives \eqref{eq:node-derivative-main}. The proof is complete.
\end{proof}

The preceding endpoint analysis is the point at which the singular problem
departs most sharply from the regular finite-interval theory.  Formula
\eqref{eq:node-derivative-main} is obtained only after proving
\eqref{eq:Wronskian-origin}; in a regular problem the corresponding boundary
term is disposed of directly by the imposed endpoint condition.  Here its
vanishing is a theorem about the Friedrichs asymptotics of both the
eigenfunction and its potential derivative.  Thus the nodal calculus below is
a genuinely singular-endpoint result, not a regular formula applied away from
$r=0$ and then extended formally.

We now summarize the weak sequential continuity and Fr\'echet differentiability of the fixed-sector nodal maps.
\begin{theorem}\label{thm:node-calculus}
Under \eqref{eq:main-p-assumption}, the map
\(T_{i,m}:L_d^p\longrightarrow(0,R)\)
is weakly sequentially continuous and continuously Fr\'echet differentiable.  If $q_n\weakto q$ in $L_d^p$, then
\(T_{i,m}(q_n)\longrightarrow T_{i,m}(q).\)
Its derivative is given by \eqref{eq:node-derivative-main}, with
$g_{i,m}(\cdot;q)\in L_d^{p'}$.  In particular,
$D T_{i,m}(q)\ne0$ and
\begin{equation}\label{eq:translation-orthogonality}
 \int_0^R g_{i,m}(r;q)\dd\mu_d(r)=0,
\end{equation}
which agrees with the translation invariance $T_{i,m}(q+c)=T_{i,m}(q)$.
\end{theorem}

\begin{proof}
The derivative formula follows from Lemma~\ref{lem:node-derivative}.  Since
$E\in \mathcal H_{\ell,0}\hookrightarrow L_d^{2p'}$, we have $E^2\in L_d^{p'}$, and therefore $g_{i,m}(\cdot;q)\in L_d^{p'}$.  Because $A,B,\Gamma>0$, the derivative is nonzero.  Moreover,
\[
 \int_0^R g_{i,m}\dd\mu_d
 =\frac{B}{\Gamma}A-\frac{A}{\Gamma}B=0.
\]
We next verify continuity of the derivative in operator norm.
Suppose $q_n\to q$ strongly in $L_d^p$, and use locally sign-fixed normalized eigenfunctions.
By Proposition~\ref{prop:eigenpair-differentiability} and the implicit-function argument in
Lemma~\ref{lem:node-derivative}, one gets
\[
 E_n\to E\text{ in }L_d^{2p'},
 \qquad
 T_n\to T,
 \qquad
 E_n'(T_n)\to E'(T),
\]
where $T=T_{i,m}(q)$ and $E=E_m(\cdot;q)$.
Hence $A_n\to A$, $B_n\to B$, and $\Gamma_n\to\Gamma$.  Furthermore,
\begin{align*}
 &\|E_n^2\ind_{(0,T_n)}-E^2\ind_{(0,T)}\|_{L_d^{p'}}\\
 &\quad\le
 \|E_n^2-E^2\|_{L_d^{p'}}
 +\|E^2(\ind_{(0,T_n)}-\ind_{(0,T)})\|_{L_d^{p'}}
 \longrightarrow0,
\end{align*}
and the same holds on the right interval.  Thus
$g_{i,m}(\cdot;q_n)\to g_{i,m}(\cdot;q)$ in $L_d^{p'}$, proving continuous Fr\'echet differentiability.

Finally, let $q_n\weakto q$.  Choose disjoint intervals around the $m-1$ simple zeros of $E_m(\cdot;q)$, all contained in $[a,R)$ for some $a>0$.  By Proposition~\ref{prop:eigenpair-compactness}, the sign-fixed eigenfunctions converge in $C^1([a,R])$.  Each simple zero therefore persists uniquely in its chosen interval.  Since every $E_m(\cdot;q_n)$ has exactly $m-1$ interior zeros, no additional zero can occur outside these intervals.  Their ordering is preserved, and hence
$T_{i,m}(q_n)\to T_{i,m}(q)$. The proof is complete.
\end{proof}

\begin{corollary}\label{cor:submersion}
For every $q\in L_d^p$, the level set
\[
 \{\widetilde q:T_{i,m}(\widetilde q)=T_{i,m}(q)\}
\]
is, locally near $q$, a $C^1$ submanifold of codimension one.  Furthermore, $T_{i,m}$ maps a neighborhood of $q$ onto an open interval containing $T_{i,m}(q)$.
\end{corollary}
\begin{proof}
The derivative is a nonzero bounded linear functional by Theorem~\ref{thm:node-calculus}.
The Banach-space submersion theorem applies.  Explicitly, with $g=g_{i,m}(\cdot;q)$, choose
\[
 h=J_{p'}(g)=|g|^{p'-2}g\in L_d^p.
\]
Then
\[
 D T_{i,m}(q)[h]=\int_0^R |g|^{p'}\dd\mu_d>0,
\]
so the scalar function $s\mapsto T_{i,m}(q+sh)$ is locally invertible.
\end{proof}

\subsection{Existence of optimal potentials}

We next use the preceding continuity result to prove existence of optimal potentials.

\begin{theorem}\label{thm:existence}
Assume \eqref{eq:main-p-assumption}.  For every observed node $T_*\in(0,R)$, the set $\mathcal S_{T_*}$ given in \eqref{eq:admissible-set} is nonempty and problem \eqref{eq:RINP} possesses at least one optimal potential $\widehat q\in\mathcal S_{T_*}$.  Moreover, for every $q_0\in L_d^p$, nearby nodal level sets form a $C^1$ codimension-one foliation of a neighborhood of $q_0$.
\end{theorem}

\begin{proposition}\label{prop:global-feasibility}
For every $m\geq2$, $1\leq i\leq m-1$, and $T_*\in(0,R)$ there exists a
bounded radial piecewise constant potential $q_*$ such that
\(T_{i,m}(q_*)=T_*\).
More generally, every ordered vector
$0<\tau_1<\cdots<\tau_{m-1}<R$ is the complete nodal vector of the $m$th
eigenfunction in the fixed sector for some bounded piecewise constant
potential.
\end{proposition}

\begin{proof}
Set $r_0=0$, $r_m=R$, and choose
$0<r_1<\cdots<r_{m-1}<R$, with $r_i=T_*$ in the one-node case.  On
$(0,r_1)$ take the positive first eigenfunction of
\[
 -\big(r^{d-1}\psi'\big)'+\gamma r^{d-3}\psi
 =\kappa_1^2r^{d-1}\psi
\]
with the Friedrichs condition at $0$ and the Dirichlet condition at $r_1$.
On each $(r_{j-1},r_j)$, $j\geq2$, take the positive first Dirichlet
eigenfunction $\phi_j$ of the same fixed-sector expression, with first
eigenvalue $\kappa_j^2$; denote the first-interval eigenfunction by
$\phi_1$.  Choose $c_1=1$ and recursively set
\begin{equation}\label{eq:gluing-coefficients}
 c_{j+1}:=c_j\frac{\phi_j'(r_j^-)}{\phi_{j+1}'(r_j^+)},
 \qquad 1\le j\le m-1.
\end{equation}
Because a positive first Dirichlet eigenfunction has negative derivative at
its right endpoint and positive derivative at its left endpoint, the signs
of the $c_j$ alternate.  The function
$\Psi=c_j\phi_j$ on $(r_{j-1},r_j)$ is continuous, vanishes at every
interface, and satisfies \(\Psi'(r_j^-)=\Psi'(r_j^+)\).
Thus its weighted flux $r^{d-1}\Psi'$ also matches, and integration by parts
on the subintervals shows that $\Psi$ is a global Friedrichs weak solution.
For any $\lambda_*\in\mathbb R$, set
\[
 q_*(r)=\lambda_*-\kappa_j^2,
 \qquad r\in(r_{j-1},r_j).
\]
The glued solution has exactly the simple zeros
$r_1,\ldots,r_{m-1}$.  Fixed-sector Sturm oscillation therefore identifies it
as an $m$th eigenfunction, proving both assertions.
\end{proof}

\begin{proof}[Proof of Theorem~\ref{thm:existence}]
By Proposition~\ref{prop:global-feasibility}, $\mathcal S_{T_*}$ is nonempty.  Choose a minimizing sequence $q_n\in\mathcal S_{T_*}$.  It is bounded in the reflexive space $L_d^p$, hence a subsequence converges weakly to some $\widehat q\in L_d^p$.  By Theorem~\ref{thm:node-calculus},
\[
 T_{i,m}(\widehat q)=\lim_{n\to\infty}T_{i,m}(q_n)=T_*,
\]
so $\widehat q\in\mathcal S_{T_*}$.  Weak lower semicontinuity gives
\[
 \|\widehat q-q_0\|_{L_d^p}
 \le\liminf_{n\to\infty}\|q_n-q_0\|_{L_d^p},
\]
and $\widehat q$ is a minimizer.  The local foliation statement is Corollary~\ref{cor:submersion}.
\end{proof}

\subsection{The critical equation for optimal potentials}

We now derive the Euler--Lagrange equation associated with the constrained minimization problem.
For $1<s<\infty$, define

\[
 \varphi_s(t):=|t|^{s-2}t.
\]

\begin{theorem}\label{thm:critical-system}
Let $\widehat q$ be a nontrivial minimizer of \eqref{eq:RINP}, namely $\widehat q\ne q_0$.  Put $T=T_*$, $E=E_m(\cdot;\widehat q)$, and let $A,B,\Gamma$ be defined by \eqref{eq:A-B-Gamma}.  Then there exist $\kappa\ne0$ and $\varepsilon=\sgn\kappa\in\{-1,1\}$ such that
\begin{equation}\label{eq:Lagrange-kernel}
 \varphi_p(\widehat q-q_0)=\kappa g_{i,m}(\cdot;\widehat q).
\end{equation}
Define the piecewise scaled eigenfunction
\begin{equation}\label{eq:scaled-U}
 U(r):=
 \begin{cases}
 \displaystyle\left(\frac{|\kappa|B}{\Gamma}\right)^{1/2}E(r),&0<r<T,\\[2mm]
 \displaystyle\left(\frac{|\kappa|A}{\Gamma}\right)^{1/2}E(r),&T<r<R.
 \end{cases}
\end{equation}
Then
\begin{equation}\label{eq:q-reconstruction-U}
 \widehat q(r)-q_0(r)=
 \begin{cases}
 \varepsilon |U(r)|^{2p'-2},&0<r<T,\\
 -\varepsilon |U(r)|^{2p'-2},&T<r<R,
 \end{cases}
\end{equation}
and $U$ satisfies, weakly on each side of $T$ and hence almost everywhere there,
\begin{equation}\label{eq:critical-sector-system}
 \begin{cases}
 -U''-\dfrac{d-1}{r}U'+\dfrac{\gamma}{r^2}U+q_0(r)U+\varepsilon\varphi_{2p'}(U)=\lambda_m(\widehat q)U,
 &0<r<T,\\[2mm]
 -U''-\dfrac{d-1}{r}U'+\dfrac{\gamma}{r^2}U+q_0(r)U-\varepsilon\varphi_{2p'}(U)=\lambda_m(\widehat q)U,
 &T<r<R.
 \end{cases}
\end{equation}
Moreover,
\begin{equation}\label{eq:critical-sector-endpoint}
 r^{-\ell}U(r)\longrightarrow a,\qquad
 r^{1-\ell}U'(r)\longrightarrow\ell a
 \quad(r\downarrow0),\qquad U(T)=U(R)=0.
\end{equation}
\begin{equation}\label{eq:mass-balance-main}
 \int_0^T U(r)^2\dd\mu_d(r)=\int_T^R U(r)^2\dd\mu_d(r),
\end{equation}
and
\begin{equation}\label{eq:norm-U-main}
 \|\widehat q-q_0\|_{L_d^p}^p=\int_0^R |U(r)|^{2p'}\dd\mu_d(r).
\end{equation}
The function $U$ has $i-1$ zeros in $(0,T)$ and $m-i-1$ zeros in $(T,R)$.  In general, $U'$ has a jump at $T$.
\end{theorem}
\begin{proof}
It is convenient to minimize \(\mathcal J(q):=\frac1p\|q-q_0\|_{L_d^p}^p\). If $\widehat q\ne q_0$ is a constrained minimizer, the constraint is regular by Theorem~\ref{thm:node-calculus}.  Therefore there exists $\kappa\in\R\setminus\{0\}$ such that
\begin{equation}\label{eq:Lagrange-abstract}
 D\mathcal J(\widehat q)=\kappa D T_{i,m}(\widehat q).
\end{equation}
Since
\[
 D\mathcal J(q)[h]=\int_0^R\varphi_p(q-q_0)h\dd\mu_d,
\]
\eqref{eq:Lagrange-abstract} is precisely \eqref{eq:Lagrange-kernel}.  On $(0,T)$ identity \eqref{eq:Lagrange-kernel} reads
\[
 \varphi_p(\widehat q-q_0)
 =\kappa\frac{B}{\Gamma}E^2,
\]
while on $(T,R)$ it reads
\[
 \varphi_p(\widehat q-q_0)
 =-\kappa\frac{A}{\Gamma}E^2.
\]
Since $\varphi_p^{-1}=\varphi_{p'}$,
\begin{align*}
 \widehat q-q_0
 &=\varepsilon\left(\frac{|\kappa|B}{\Gamma}\right)^{p'-1}|E|^{2p'-2},&&0<r<T,\\
 \widehat q-q_0
 &=-\varepsilon\left(\frac{|\kappa|A}{\Gamma}\right)^{p'-1}|E|^{2p'-2},&&T<r<R.
\end{align*}
Definition \eqref{eq:scaled-U} gives \eqref{eq:q-reconstruction-U}.  Multiplying the eigenvalue equation for $E$ by the corresponding constant scaling factor on each side of $T$ gives \eqref{eq:critical-sector-system}.  The endpoint, boundary, and nodal conditions follow from those of $E$.

For the balance law, it follows from \eqref{eq:scaled-U} that
\begin{align*}
 \int_0^T U^2\dd\mu_d
 &=\frac{|\kappa|B}{\Gamma}\int_0^T E^2\dd\mu_d
 =\frac{|\kappa|AB}{\Gamma},\\
 \int_T^R U^2\dd\mu_d
 &=\frac{|\kappa|A}{\Gamma}\int_T^R E^2\dd\mu_d
 =\frac{|\kappa|AB}{\Gamma}.
\end{align*}
Finally,
\[
 |\widehat q-q_0|^p=|U|^{(2p'-2)p}=|U|^{2p'},
\]
which proves \eqref{eq:norm-U-main}.  The nodal count is inherited from $E$. The proof is complete.
\end{proof}

\begin{remark}\label{rem:transmission}
Although $U(T)=0$, the derivative generally jumps because the two scaling factors in \eqref{eq:scaled-U} differ.  More precisely,

\[
 \frac{U'(T^-)}{U'(T^+)}=\sqrt{\frac BA}.
\]
The original eigenfunction $E$ remains $C^1$ across $T$.  This distinction is important in the shooting reconstruction of Section~\ref{sec:shooting}.
\end{remark}

\begin{remark}\label{rem:trivial}
If $T_*=T_{i,m}(q_0)$, then $q_0$ is the unique minimizer because the minimum distance is zero.  The nonlinear critical system is only needed for $T_*\ne T_{i,m}(q_0)$.
\end{remark}

\section{Reconstruction of minimal fixed-sector potentials}\label{sec:reconstruction}

Throughout this section, for a fixed reference potential $q_0$, we write
\[
 E_m^0:=E_m(\cdot;q_0),\qquad T_0:=T_{i,m}(q_0),
\]
\[
 A_0:=\int_0^{T_0}|E_m^0|^2\dd\mu_d,\qquad
 B_0:=\int_{T_0}^R|E_m^0|^2\dd\mu_d=1-A_0,
\]
\[
 \Gamma_0:=T_0^{d-1}\big((E_m^0)'(T_0)\big)^2,
 \qquad g_0:=g_{i,m}(\cdot;q_0).
\]

\subsection{Bessel reference nodes and fixed-sector dynamics}

Assume that $q_0(r)\equiv c$. Let $J_\nu$ denote the Bessel function of the first kind of order $\nu$,
 i.e., the solution regular at the origin of
\[
x^2 y'' + x y' + (x^2-\nu^2)y = 0,
\]
where $\nu=\ell+\frac{d-2}{2}$ is defined in \eqref{eq:fixed-sector-parameters},
and let $j_{\nu,k}$ denote its $k$-th positive zero.
With $C_m$ chosen so that $\|E_m^0\|_{L_d^2}=1$, we have
\begin{equation}\label{eq:Bessel-prior}
 E_m^0(r)=C_m r^{-(d-2)/2}
 J_\nu\left(\frac{j_{\nu,m}}Rr\right),
 \qquad
 T_0=R\frac{j_{\nu,i}}{j_{\nu,m}},
\end{equation}
and
\begin{equation}\label{eq:Bessel-eigenvalue}
 \lambda_m(q_0)=c+\frac{j_{\nu,m}^2}{R^2}.
\end{equation}
Combining \eqref{eq:Bessel-prior} with
\eqref{eq:node-derivative-main} gives
\[
 g_0(r)=
 \frac{B_0}{\Gamma_0}|E_m^0(r)|^2\ind_{(0,T_0)}(r)
 -\frac{A_0}{\Gamma_0}|E_m^0(r)|^2\ind_{(T_0,R)}(r).
\]
Thus the Bessel kernel determines the leading-order optimal reconstruction
in Theorem~\ref{thm:local-asymptotic}.

The centrifugal term changes the radial energy identity as follows.
The decisive distinction from the one-dimensional constant-prior problem is the following identity.
\begin{proposition}\label{prop:energy-dissipation}
Assume $q_0\equiv c$ and let $U$ solve one side of
\eqref{eq:critical-sector-system}.  Write
$\lambda:=\lambda_m(\widehat q)$ and let $\varsigma=1$ on the left interval
and $\varsigma=-1$ on the right interval.  Define
\begin{equation}\label{eq:sector-energy}
 \mathcal E_\varsigma(r)
 :=\frac12U'(r)^2
 +\frac12\left(\lambda-c-\frac{\gamma}{r^2}\right)U(r)^2
 -\frac{\varsigma\varepsilon}{2p'}|U(r)|^{2p'}.
\end{equation}
Then
\begin{equation}\label{eq:energy-dissipation}
 \mathcal E_\varsigma'(r)
 =-\frac{d-1}{r}U'(r)^2+\frac{\gamma}{r^3}U(r)^2.
\end{equation}
In particular, the energy is nonincreasing for $\ell=0$.
\end{proposition}

\begin{proof}
Differentiate \eqref{eq:sector-energy} and use
\eqref{eq:critical-sector-system}.  The derivative of
$-\gamma U^2/(2r^2)$ produces the term $\gamma U^2/r^3$, and all remaining
terms cancel as stated.
\end{proof}

\begin{remark}\label{rem:Liouville-weight}
The Liouville substitution removes the first derivative but retains the
inverse-square Bessel term and transforms the physical $L_d^p$ nonlinearity
into an explicitly weighted power.  It therefore does not produce an
autonomous one-dimensional optimization problem.  When $\ell=0$,
\eqref{eq:energy-dissipation} reduces to the monotone dissipative identity;
for higher angular momentum even this monotonicity is lost.
\end{remark}

\subsection{A finite-dimensional observation principle}\label{sec:abstract-observation}
The local reconstruction arguments used below all have the same functional
analytic core.  We record it once in a form that can be applied to a single
node, several nodes, mixed angular-momentum data, and augmented spectral--nodal
observations.

\begin{proposition}
\label{prop:finite-observation-principle}
Let $X=L_d^p$, $1<p<\infty$, and let
$\mathcal O:X\to\mathbb R^M$ be weakly sequentially continuous on $X$ and
continuously Fr\'echet differentiable in a neighborhood of $q_0$.  Put
\[
 \mathcal O_0:=\mathcal O(q_0),\qquad
 G_0:=D\mathcal O(q_0):X\to\mathbb R^M,
\]
and assume that $G_0$ is surjective.  For $\boldsymbol y\in\mathbb R^M$ define
\begin{equation}\label{eq:abstract-rho-R}
 \rho_0(\boldsymbol y)
 :=\min\{\|h\|_X:G_0h=\boldsymbol y\},\qquad
 \mathscr R_0(\boldsymbol y)
 :=\operatorname*{argmin}_{G_0h=\boldsymbol y}\|h\|_X.
\end{equation}
Then $\mathscr R_0$ is well defined, continuous, odd, and homogeneous of
degree one, and $\rho_0$ is a norm equivalent to the Euclidean norm.
Moreover, as $\boldsymbol y\to0$:
\begin{enumerate}
\renewcommand{\labelenumi}{\textup{(\roman{enumi})}}
\item there exists an exactly feasible comparison potential
$q_{\boldsymbol y}^{\rm c}$ such that
\begin{equation}\label{eq:abstract-feasible-correction}
 \mathcal O(q_{\boldsymbol y}^{\rm c})=\mathcal O_0+\boldsymbol y,
 \qquad
 q_{\boldsymbol y}^{\rm c}-q_0
 =\mathscr R_0(\boldsymbol y)+o_X(|\boldsymbol y|);
\end{equation}
\item the minimum-distance problem
\begin{equation}\label{eq:abstract-min-problem}
 \min\{\|q-q_0\|_X:\mathcal O(q)=\mathcal O_0+\boldsymbol y\}
\end{equation}
has a minimizer $\widehat q_{\boldsymbol y}$, and every such minimizer satisfies
\begin{align}
 \widehat q_{\boldsymbol y}-q_0
 &=\mathscr R_0(\boldsymbol y)+o_X(|\boldsymbol y|),
 \label{eq:abstract-min-asymptotic}\\
 \|\widehat q_{\boldsymbol y}-q_0\|_X
 &=\rho_0(\boldsymbol y)+o(|\boldsymbol y|).
 \label{eq:abstract-distance-asymptotic}
\end{align}
\item if $p=2$ and $\mathcal O$ is $C^2$ near $q_0$, then
\begin{equation}\label{eq:abstract-Gram}
 \mathbb M_0:=G_0G_0^*
\end{equation}
is positive definite,
\begin{equation}\label{eq:abstract-Hilbert-right-inverse}
 \mathscr R_0(\boldsymbol y)
 =G_0^*\mathbb M_0^{-1}\boldsymbol y,
\end{equation}
and for all sufficiently small $\boldsymbol y$ the minimizer is unique and
depends $C^1$ on $\boldsymbol y$.  In addition,
\begin{align}
 \widehat q_{\boldsymbol y}-q_0
 &=G_0^*\mathbb M_0^{-1}\boldsymbol y
   +O_X(|\boldsymbol y|^2),
 \label{eq:abstract-Hilbert-expansion}\\
 \|\widehat q_{\boldsymbol y}-q_0\|_X^2
 &=\boldsymbol y^{\mathsf T}\mathbb M_0^{-1}\boldsymbol y
   +O(|\boldsymbol y|^3).
 \label{eq:abstract-Hilbert-cost}
\end{align}
\end{enumerate}
\end{proposition}
\begin{proof}
For each $\boldsymbol y$, the affine set
$\{h:G_0h=\boldsymbol y\}$ is nonempty and weakly closed.  Reflexivity of
$X$ gives existence of a minimum-norm element and strict convexity gives its
uniqueness.  Thus $\rho_0$ is the quotient norm induced by $G_0$, hence is
equivalent to the Euclidean norm on $\mathbb R^M$; uniqueness also gives the
oddness, homogeneity, and continuity of $\mathscr R_0$.

Choose a bounded linear right inverse $L:\mathbb R^M\to X$ of $G_0$.  Put
\[
 r_{\boldsymbol y}
 :=\mathcal O(q_0+\mathscr R_0(\boldsymbol y))
   -\mathcal O_0-\boldsymbol y=o(|\boldsymbol y|).
\]
For
$\Psi_{\boldsymbol y}(\boldsymbol x)
 :=\mathcal O(q_0+\mathscr R_0(\boldsymbol y)+L\boldsymbol x)
   -\mathcal O_0-\boldsymbol y$
we have $D_{\boldsymbol x}\Psi_{\boldsymbol y}(0)\to G_0L=I$ as
$\boldsymbol y\to0$. Here are the details needed for a
uniform correction.  Define
\[
 \mathcal K_{\boldsymbol y}(\boldsymbol x)
 :=\boldsymbol x-\Psi_{\boldsymbol y}(\boldsymbol x).
\]
Because $D\mathcal O$ is continuous and
$\mathscr R_0(\boldsymbol y)\to0$, there exists a fixed neighborhood of the
origin such that, for all sufficiently small $\boldsymbol y$,
\[
 \|D\mathcal K_{\boldsymbol y}(\boldsymbol x)\|
 =\|I-D\mathcal O(q_0+\mathscr R_0(\boldsymbol y)
                  +L\boldsymbol x)L\|\le\frac12
\]
throughout that neighborhood.  Since
$\mathcal K_{\boldsymbol y}(0)=-r_{\boldsymbol y}$, the map
$\mathcal K_{\boldsymbol y}$ sends the ball
$\overline B(0,2|r_{\boldsymbol y}|)\subset\mathbb R^M$ into itself and is a
contraction there once $\boldsymbol y$ is small.  Its unique fixed point
$\boldsymbol x_{\boldsymbol y}$ satisfies
\[
 |\boldsymbol x_{\boldsymbol y}|\le2|r_{\boldsymbol y}|
 =o(|\boldsymbol y|),
 \qquad
 \Psi_{\boldsymbol y}(\boldsymbol x_{\boldsymbol y})=0.
\]
Thus
$q_{\boldsymbol y}^{\rm c}:=q_0+\mathscr R_0(\boldsymbol y)
+L\boldsymbol x_{\boldsymbol y}$ satisfies
\eqref{eq:abstract-feasible-correction}.

Weak sequential continuity makes each nearby level set weakly sequentially
closed, so the direct method gives a minimizer of
\eqref{eq:abstract-min-problem}.  The feasible comparison above shows that
$h_{\boldsymbol y}:=\widehat q_{\boldsymbol y}-q_0$ satisfies
$\|h_{\boldsymbol y}\|_X=O(|\boldsymbol y|)$.  Fr\'echet differentiability gives
\[
 \boldsymbol y=G_0h_{\boldsymbol y}+o(\|h_{\boldsymbol y}\|_X).
\]
The quotient norm therefore yields
$\rho_0(\boldsymbol y)\le\|h_{\boldsymbol y}\|_X+o(|\boldsymbol y|)$,
while the comparison potential gives the reverse upper estimate.  Hence
\eqref{eq:abstract-distance-asymptotic} holds.  If
$\boldsymbol y_n\to0$ and
$\boldsymbol e_n=\boldsymbol y_n/\rho_0(\boldsymbol y_n)$,
$w_n=h_{\boldsymbol y_n}/\rho_0(\boldsymbol y_n)$, then every weak limit of
$w_n$ is a unit minimum-norm solution of $G_0w=\lim\boldsymbol e_n$ and is
therefore $\mathscr R_0(\lim\boldsymbol e_n)$.  The Kadec--Klee property of
$L_d^p$ upgrades the convergence to strong convergence, proving
\eqref{eq:abstract-min-asymptotic}.

Now let $p=2$.  Orthogonal decomposition gives
$\mathbb M_0=G_0G_0^*>0$ and
\eqref{eq:abstract-Hilbert-right-inverse}.  Since surjectivity persists near
$q_0$, every sufficiently small-data minimizer satisfies the KKT system
\[
 h-D\mathcal O(q_0+h)^*\boldsymbol\lambda=0,
 \qquad
 \mathcal O(q_0+h)-\mathcal O_0-\boldsymbol y=0.
\]
Its linearization in $(h,\boldsymbol\lambda)$ at the origin is
\[
 (v,\boldsymbol\mu)\longmapsto
 \bigl(v-G_0^*\boldsymbol\mu,\,G_0v\bigr),
\]
which is an isomorphism because $\mathbb M_0$ is positive definite.  The
implicit function theorem gives a unique local critical branch.  By
\eqref{eq:abstract-distance-asymptotic}, every global minimizer lies in this
neighborhood; the uniform lower bound for
$D\mathcal O(q)^*$ also gives $|\boldsymbol\lambda|=O(|\boldsymbol y|)$.
Thus all minimizers coincide with the IFT branch.  A second-order expansion of
the KKT system yields
$\boldsymbol\lambda=\mathbb M_0^{-1}\boldsymbol y+O(|\boldsymbol y|^2)$ and
\eqref{eq:abstract-Hilbert-expansion}; squaring gives
\eqref{eq:abstract-Hilbert-cost}.
\end{proof}

\subsection{The direct relationship between the observed node and the minimal distance}
For the general reference potential fixed above, set
\[
 J_{p'}(g_0):=|g_0|^{p'-2}g_0.
\]

\begin{theorem}\label{thm:local-asymptotic}
Assume \eqref{eq:main-p-assumption}.  For sufficiently small $\delta$, let $\widehat q_\delta$ be any minimizer satisfying
\[
 T_{i,m}(\widehat q_\delta)=T_0+\delta.
\]
Then, as $\delta\to0$,
\begin{equation}\label{eq:local-reconstruction-main}
 \widehat q_\delta-q_0
 =\frac{\delta}{\|g_0\|_{L_d^{p'}}^{p'}}J_{p'}(g_0)+o_{L_d^p}(|\delta|),
\end{equation}
and
\begin{equation}\label{eq:distance-main}
 \|\widehat q_\delta-q_0\|_{L_d^p}
 =\frac{|\delta|}{\|g_0\|_{L_d^{p'}}}+o(|\delta|).
\end{equation}
Consequently, to leading order, a rightward displacement of the node increases the optimal potential on $(0,T_0)$ and decreases it on $(T_0,R)$; a leftward displacement reverses these signs.
\end{theorem}

\begin{proof}[Proof of Theorem~\ref{thm:local-asymptotic}]
Apply Proposition~\ref{prop:finite-observation-principle} with
$M=1$ and $\mathcal O(q)=T_{i,m}(q)$.  Theorem~\ref{thm:node-calculus}
gives weak sequential continuity and $C^1$ regularity, while
$G_0h=\int_0^R g_0h\,\dd\mu_d$ is onto because $g_0\ne0$.  H\"older's
inequality shows that the unique minimum-norm solution of $G_0h=\delta$ is
\[
 \mathscr R_0(\delta)
 =\frac{\delta}{\|g_0\|_{L_d^{p'}}^{p'}}J_{p'}(g_0),
 \qquad
 \rho_0(\delta)=\frac{|\delta|}{\|g_0\|_{L_d^{p'}}}.
\]
Equations \eqref{eq:local-reconstruction-main} and
\eqref{eq:distance-main} are therefore exactly
\eqref{eq:abstract-min-asymptotic} and
\eqref{eq:abstract-distance-asymptotic}.  The sign pattern follows from the
explicit two-step formula \eqref{eq:node-derivative-main} for $g_0$.
\end{proof}

\begin{corollary}\label{cor:local-sign}
Let $\widehat q_\delta$ be a nontrivial optimal potential and let $\varepsilon_\delta$ be the sign in Theorem~\ref{thm:critical-system}.
For all sufficiently small $\delta\ne0$,
\(\varepsilon_\delta=\sgn\delta\).
\end{corollary}

\begin{proof}
Choose a compact interval $K\Subset(0,T_0)$ which avoids all zeros of
$E_m(\cdot;q_0)$ and has positive measure.  Then $g_0>0$ on $K$ and, after shrinking $K$ if necessary,
\[
 g_0(r)=\frac{B_0}{\Gamma_0} \left| E_m(r; q_0) \right|^2\ge c_K>0\qquad\text{for }r\in K.
\]
For all sufficiently small $|\delta|$, we have $K\subset(0,T_0+\delta)$.  Suppose, along a sequence $\delta_n\to0$, that
$\varepsilon_{\delta_n}=-\sgn\delta_n$.  By
\eqref{eq:q-reconstruction-U},
\[
 \sgn\delta_n\,
 (\widehat q_{\delta_n}-q_0)\le0
 \quad\text{a.e. on }K.
\]
On the other hand, \eqref{eq:local-reconstruction-main} implies
\[
 \frac{\sgn\delta_n}{|\delta_n|}
 (\widehat q_{\delta_n}-q_0)
 \longrightarrow
 \frac{J_{p'}(g_0)}{\|g_0\|_{L_d^{p'}}^{p'}}
 \quad\text{in }L_d^p(K),
\]
and the limit is bounded below by a positive constant on $K$.
This is impossible for a sequence which is nonpositive almost everywhere on $K$.  Hence
$\varepsilon_\delta=\sgn\delta$ for all sufficiently small nonzero $\delta$.
\end{proof}

\subsection{The uniqueness of the optimal potential}

\begin{lemma}\label{lem:nodal-flux-C1}
Let $d\in\{2,3\}$ and $p=2$.  Use the locally sign-fixed, $L_d^2$-normalized eigenfunction supplied
by Proposition~\ref{prop:eigenpair-differentiability}, and write
\[
 T=T_{i,m}(q),\qquad E=E_m(\cdot;q),\qquad
 \Pi(q):=T_{i,m}(q)^{d-1}E_m'(T_{i,m}(q);q).
\]
Then $E'$ is differentiable with respect to the radial variable at the simple node $T$, in the difference-quotient sense, and
\begin{equation}\label{eq:nodal-curvature}
 E''(T)=-\frac{d-1}{T}E'(T).
\end{equation}
Moreover, $\Pi:L_d^2\to\R$ is continuously Fr\'echet differentiable. If $Z=DE_m(q)[h]$, then
\begin{equation}\label{eq:nodal-flux-derivative}
 D\Pi(q)[h]=T^{d-1}Z'(T).
\end{equation}
Consequently, the nodal slope $S(q):=E_m'(T_{i,m}(q);q)$ and the quantity
$\Gamma(q)=T_{i,m}(q)^{d-1}S(q)^2$ are $C^1$, with
\begin{align}
 DS(q)[h]
 &=Z'(T)-\frac{d-1}{T}E'(T)DT_{i,m}(q)[h],\label{eq:nodal-slope-derivative}\\
 D\Gamma(q)[h]
 &=2T^{d-1}E'(T)Z'(T)
 -(d-1)T^{d-2}E'(T)^2DT_{i,m}(q)[h].
 \label{eq:Gamma-derivative}
\end{align}
\end{lemma}

\begin{proof}
Because $T>0$, the weighted and unweighted $L^2$ norms are equivalent on a fixed neighborhood of $T$.  By local $W^{2,2}$ regularity, $E\in C^1$ there, and the simple-node identity $E(T)=0$ gives $|E(r)|\le C|r-T|$.  Integrating the radial equation between $T$ and $T+\eta$ gives
\begin{align*}
 E'(T+\eta)-E'(T)
 &=-\int_T^{T+\eta}\frac{d-1}{r}E'(r)\dd r
 +\int_T^{T+\eta}
 \left(\frac{\gamma}{r^2}+q(r)-\lambda\right)E(r)\dd r.
\end{align*}
After division by $\eta$, the first term converges to $-(d-1)E'(T)/T$.  For the potential term, with $I_\eta$ the interval between $T$ and $T+\eta$,
\[
 \frac1{|\eta|}\left|\int_{I_\eta}qE\dd r\right|
 \le C|\eta|^{1/2}\|q\|_{L^2(I_\eta)}\longrightarrow0,
\]
while the centrifugal and $\lambda E$ contributions are both $O(|\eta|)$.  This proves \eqref{eq:nodal-curvature}.
For $h\to0$ in $L_d^2$, set
\begin{align*}
 E_h&=E_m(\cdot;q+h), & T_h&=T_{i,m}(q+h),
 & \lambda_h&=\lambda_m(q+h),\\
 P_h(r)&=r^{d-1}E_h'(r), & P_0(r)&=r^{d-1}E'(r).
\end{align*}
and let $Z=DE_m(q)[h]$.  The local $W^{2,2}$ differentiability in
Proposition~\ref{prop:eigenpair-differentiability} gives
\begin{equation}\label{eq:fixed-flux-expansion}
 P_h(T)-P_0(T)-T^{d-1}Z'(T)=o(\|h\|_{L_d^2}).
\end{equation}
Since $T_h-T=O(\|h\|_{L_d^2})$, the eigenfunctions are uniformly $C^1$ near $T$, and $E_h(T_h)=0$ implies $|E_h(r)|\le C|r-T_h|$.  The flux equation yields
\[
 P_h(T_h)-P_h(T)
 =\int_T^{T_h}
 \left[\gamma s^{d-3}+s^{d-1}(q(s)+h(s)-\lambda_h)\right]
 E_h(s)\dd s.
\]
If $I_h$ is the interval between $T$ and $T_h$, then
\begin{align*}
 |P_h(T_h)-P_h(T)|
 &\le C\|q+h\|_{L^2(I_h)}|T_h-T|^{3/2}
 +C|T_h-T|^2\\
 &=o(\|h\|_{L_d^2}),
\end{align*}
Together with \eqref{eq:fixed-flux-expansion}, this proves
\eqref{eq:nodal-flux-derivative}.
Let $q_n\to q$ in $L_d^2$ and choose $a<T/2$.  Then $T(q_n)\to T(q)$ and
\[
 DE_m(q_n)\longrightarrow DE_m(q)
 \quad\text{in }\mathcal L(L_d^2,W^{2,2}(a,R))
\]
by Proposition~\ref{prop:eigenpair-differentiability}.  The embedding
$W^{2,2}(a,R)\hookrightarrow C^{1,1/2}([a,R])$ makes the unit images equicontinuous in their first derivatives.  Formula \eqref{eq:nodal-flux-derivative} therefore converges in operator norm, so $\Pi$ is $C^1$.  Finally,
\[
 S(q)=\frac{\Pi(q)}{T(q)^{d-1}},\qquad
 \Gamma(q)=\frac{\Pi(q)^2}{T(q)^{d-1}},
\]
and differentiation gives \eqref{eq:nodal-slope-derivative} and
\eqref{eq:Gamma-derivative}. The proof is complete.
\end{proof}

\begin{proposition}\label{prop:Hilbert-C2}
Let $d\in\{2,3\}$ and $p=2$.  For every $q\in L_d^2$, the nodal map
$q\mapsto T_{i,m}(q)$ is $C^2$, and its $L_d^2$ gradient
$q\mapsto g_{i,m}(\cdot;q)$ is a $C^1$ map from $L_d^2$ to $L_d^2$.
\end{proposition}

\begin{proof}
For $d=2,3$, the embedding $ \mathcal H_{\ell,0}\hookrightarrow L_d^4$ is continuous, so
\[
 (q,u,\phi)\longmapsto\int_0^R qu\phi\dd\mu_d
\]
is a continuous trilinear form on $L_d^2\times \mathcal H_{\ell,0}\times \mathcal H_{\ell,0}$.
The implicit-function map defining the simple normalized eigenpair is therefore smooth with values
in $\R\times \mathcal H_{\ell,0}$.  By Proposition~\ref{prop:eigenpair-differentiability},
it is also $C^1$ with values in $W^{2,2}(a,R)$ for every $a>0$.  The nodal map is already $C^1$ by
Theorem~\ref{thm:node-calculus}.
Define
\[
 \mathcal P(q):=E_m(\cdot;q)^2\ind_{(0,T_{i,m}(q))}.
\]
For an increment $h$, put $E_h=E(q+h)$, $T_h=T(q+h)$,
$Z=DE(q)[h]$, and $R_h=E_h-E-Z$.  Since the eigenpair map is differentiable in $ \mathcal H_{\ell,0}\hookrightarrow L_d^4$,
\[
 \|R_h\|_{L_d^4}=o(\|h\|_{L_d^2}),
 \qquad \|Z\|_{L_d^4}=O(\|h\|_{L_d^2}).
\]
The exact identity
\begin{align*}
 &\mathcal P(q+h)-\mathcal P(q)-2EZ\ind_{(0,T)}\\
 &\quad=(E_h^2-E^2-2EZ)\ind_{(0,T)}
 +E_h^2\big(\ind_{(0,T_h)}-\ind_{(0,T)}\big)\\
&\quad=(2ER_h+(Z+R_h)^2)\ind_{(0,T)}+E_h^2\big(\ind_{(0,T_h)}-\ind_{(0,T)}\big)
\end{align*}
holds for either sign of $T_h-T$.  The first term is
$o(\|h\|_{L_d^2})$ in $L_d^2$.  On the moving strip between $T$ and $T_h$, the uniform local $C^1$ bound and $E_h(T_h)=0$ imply
\[
 \left\|E_h^2\big(\ind_{(0,T_h)}-\ind_{(0,T)}\big)\right\|_{L_d^2}
 \le C|T_h-T|^{5/2}
 =o(\|h\|_{L_d^2}).
\]
Consequently, \(D\mathcal P(q)[h]=2E\,DE(q)[h]\ind_{(0,T)}\). We now verify continuity in operator norm.  Let $q_n\to q$ in $L_d^2$, and write
\[
 E_n=E(q_n),\qquad T_n=T(q_n),\qquad
 Z_n[h]=DE(q_n)[h].
\]
The smooth eigenpair map in $ \mathcal H_{\ell,0}$ gives
\[
 E_n\to E\quad\hbox{in }L_d^4,
 \qquad
 DE(q_n)\to DE(q)
 \quad\hbox{in }\mathcal L(L_d^2,L_d^4).
\]
On the common part of $(0,T_n)$ and $(0,T)$, H\"older's inequality therefore yields
\begin{align*}
 &\sup_{\|h\|_{L_d^2}\le1}
 \|E_nZ_n[h]-EZ[h]\|_{L_d^2}\\
 &\quad\le
 \|E_n-E\|_{L_d^4}
 \sup_{\|h\|_{L_d^2}\le1}\|Z_n[h]\|_{L_d^4}
 +\|E\|_{L_d^4}
 \|DE(q_n)-DE(q)\|_{\mathcal L(L_d^2,L_d^4)}
 \longrightarrow0.
\end{align*}
Choose a compact interval $K\Subset(0,R)$ containing $T_n$ and $T$ for all large $n$.  The local $W^{2,2}$ estimates imply the uniform bounds
\[
 |E_n(r)|\le C|r-T_n|,
 \qquad
 \|Z_n[h]\|_{L^\infty(K)}
 \le C\|h\|_{L_d^2}.
\]
If $I_n$ is the interval between $T_n$ and $T$, then
\[
 \sup_{\|h\|_{L_d^2}\le1}
 \|E_nZ_n[h]\ind_{I_n}\|_{L_d^2}
 \le C|T_n-T|^{3/2}\longrightarrow0,
\]
and the analogous estimate holds for $EZ[h]\ind_{I_n}$.  Combining the common-interval and moving-strip estimates gives
\[
 \|D\mathcal P(q_n)-D\mathcal P(q)\|_{
 \mathcal L(L_d^2,L_d^2)}\longrightarrow0.
\]
Thus $\mathcal P:L_d^2\to L_d^2$ is $C^1$.  The right-hand cutoff $E^2\ind_{(T,R)}=E^2-\mathcal P(q)$ is $C^1$ as well.
It follows that
\[
 A(q)=\int_0^R\mathcal P(q)\dd\mu_d,
 \qquad B(q)=1-A(q)
\]
are $C^1$.  By Lemma~\ref{lem:nodal-flux-C1}, $\Gamma(q)$ is $C^1$ and remains positive.
 Formula \eqref{eq:node-derivative-main} therefore shows that
$q\mapsto g_{i,m}(\cdot;q)$ is a $C^1$ map from $L_d^2$ to $L_d^2$.
Since
\(DT_{i,m}(q)[h]=\langle g_{i,m}(\cdot;q),h\rangle_{L_d^2}\),
the Riesz representation identifies
$DT_{i,m}(q)$ with $g_{i,m}(\cdot;q)$.
By the preceding argument, \(q\longmapsto g_{i,m}(\cdot;q)\)
is a $C^1$ map from $L_d^2$ to $L_d^2$.
Since the Riesz isomorphism
\(L_d^2\longrightarrow (L_d^2)^*\)
is continuous and linear, the map
\[
q\longmapsto DT_{i,m}(q)
\]
is $C^1$ with values in
$\mathcal L(L_d^2,\mathbb R)$.
Hence $T_{i,m}$ is $C^2$. Moreover, for any $h,k\in L_d^2$
\[
D^2T_{i,m}(q)[h,k]
=
\left\langle
Dg_{i,m}(q)[h],k
\right\rangle_{L_d^2}.
\]
\end{proof}

\begin{theorem}\label{thm:Hilbert-uniqueness}
Let $d\in\{2,3\}$ and $p=2$.  For every $q_0\in L_d^2$, there exists $\eta>0$ such that, for $|\delta|<\eta$, problem \eqref{eq:RINP} with $T_*=T_0+\delta$ has a unique minimizer $q_\delta$.  The map $\delta\mapsto q_\delta$ is $C^1$ near $0$ and
\begin{equation}\label{eq:Hilbert-expansion}
 q_\delta=q_0+\delta\frac{g_0}{\|g_0\|_{L_d^2}^2}
 +O_{L_d^2}(\delta^2).
\end{equation}
\end{theorem}

\begin{proof}
Apply Proposition~\ref{prop:finite-observation-principle} with
$\mathcal O(q)=T_{i,m}(q)$ and $p=2$.  Proposition~\ref{prop:Hilbert-C2}
provides the required $C^2$ regularity, and
$G_0h=\langle g_0,h\rangle_{L_d^2}$ is onto.  Hence
$\mathbb M_0=\|g_0\|_{L_d^2}^2$ and
\[
 G_0^*\mathbb M_0^{-1}\delta
 =\delta\frac{g_0}{\|g_0\|_{L_d^2}^2}.
\]
The uniqueness, $C^1$ dependence, and the quadratic remainder in
\eqref{eq:Hilbert-expansion} follow directly from
\eqref{eq:abstract-Hilbert-expansion}.
\end{proof}

\subsection{Reconstruction through three characteristic shooting parameters}
\label{sec:shooting}

Assume in this subsection that $q_0\in C([0,R])$.  For the left-hand problem it is convenient to write the regular solution as $U=r^\ell W$.

\begin{lemma}\label{lem:regular-shooting}
Fix $\varepsilon\in\{-1,1\}$, $\lambda\in\mathbb R$, and $a\in\mathbb R$.
There exist $\rho>0$ and a unique solution $W$ on $[0,\rho]$ of
\begin{equation}\label{eq:regular-W-equation}
 -W''-\frac{d+2\ell-1}{r}W'+q_0W
 +\varepsilon r^{2\ell(p'-1)}\varphi_{2p'}(W)=\lambda W,
 \qquad W(0)=a,\quad W'(0)=0.
\end{equation}
It is equivalently characterized by
\begin{equation}\label{eq:regular-shooting-integral}
\begin{aligned}
 W(r)=a+\int_0^r s^{1-d-2\ell}\int_0^s\tau^{d+2\ell-1}
 \big[&(q_0(\tau)-\lambda)W(\tau)\\
 &+\varepsilon\tau^{2\ell(p'-1)}
 \varphi_{2p'}(W(\tau))\big]\dd\tau\dd s.
\end{aligned}
\end{equation}
The solution extends uniquely to a maximal interval and depends locally
$C^1$ on $(a,\lambda)$ while bounded.  The associated profile
$U(r)=r^\ell W(r)$ is the unique Friedrichs solution satisfying
\[
 \lim_{r\downarrow0}r^{-\ell}U(r)=a.
\]
\end{lemma}

\begin{proof}
On a bounded set of $W$-values, the nonlinear integrand is uniformly bounded
and locally Lipschitz.  The Volterra operator in
\eqref{eq:regular-shooting-integral} satisfies
\[
 \int_0^rs^{1-d-2\ell}\int_0^s\tau^{d+2\ell-1}\dd\tau\dd s
 =\frac{r^2}{2(d+2\ell)}.
\]
It is therefore a contraction on a sufficiently small interval.  Standard
Volterra continuation and parameter-dependence arguments give the remaining
claims.  Substitution $U=r^\ell W$ recovers the left equation in
\eqref{eq:critical-sector-system} and the Friedrichs asymptotics.
\end{proof}

For $a\neq0$, let $W_-(\cdot;a,\lambda)$ be the solution of
\eqref{eq:regular-W-equation} and set
\[
 U_-(r;a,\lambda):=r^\ell W_-(r;a,\lambda).
\]
For $b\neq0$, let $U_+(\cdot;b,\lambda)$ be the maximal solution obtained
backward from $R$:
\begin{equation}\label{eq:right-shooting}
 \begin{cases}
 -U_+''-\dfrac{d-1}{r}U_+'+\dfrac{\gamma}{r^2}U_+
 +q_0(r)U_+-\varepsilon\varphi_{2p'}(U_+)=\lambda U_+,\\
 U_+(R)=0,\qquad U_+'(R)=b.
 \end{cases}
\end{equation}
Let $Z_i^-(a,\lambda)$ be the $i$th positive zero of $U_-$ and let
$Z_j^+(b,\lambda)$ be the $j$th zero encountered from $R$ toward $0$,
excluding $R$.  Define
\[
 M_-(a,\lambda;T):=\int_0^TU_-(r;a,\lambda)^2\dd\mu_d,
 \qquad
 M_+(b,\lambda;T):=\int_T^RU_+(r;b,\lambda)^2\dd\mu_d.
\]

\begin{remark}
For an unnormalized eigenfunction $E$, put
$M=\int_0^RE^2\dd\mu_d$, $A=\int_0^TE^2\dd\mu_d$, and
$B=\int_T^RE^2\dd\mu_d$.  Then
\begin{equation}\label{eq:unnormalized-kernel}
 DT_{i,m}(q)[h]=\int_0^R\left[
 \frac{B}{M\Gamma}E^2\ind_{(0,T)}
 -\frac{A}{M\Gamma}E^2\ind_{(T,R)}\right]h\dd\mu_d,
 \qquad \Gamma=T^{d-1}E'(T)^2.
\end{equation}
\end{remark}

\begin{theorem}\label{thm:shooting}
Assume $q_0\in C([0,R])$.  For the sign $\varepsilon$ associated with a
nontrivial minimizer by Theorem~\ref{thm:critical-system}, every minimizer
determines, up to the independent sign symmetries of the two pieces, a triple
$(a,b,\lambda)$ satisfying
\begin{equation}\label{eq:shooting-system}
 \boxed{\begin{aligned}
 Z_i^-(a,\lambda)&=T_*,\\
 Z_{m-i}^+(b,\lambda)&=T_*,\\
 M_-(a,\lambda;T_*)&=M_+(b,\lambda;T_*).
 \end{aligned}}
\end{equation}
Here $a=\lim_{r\downarrow0}r^{-\ell}U(r)$ is the regular leading
coefficient.  Conversely, suppose that a triple solves
\eqref{eq:shooting-system}, all relevant zeros are simple, and the terminal
derivatives at $T_*$ are chosen with the same sign.  Define
\begin{equation}\label{eq:shooting-potential}
 q(r)=\begin{cases}
 q_0(r)+\varepsilon|U_-(r)|^{2p'-2},&0<r<T_*,\\
 q_0(r)-\varepsilon|U_+(r)|^{2p'-2},&T_*<r<R.
 \end{cases}
\end{equation}
Then $q$ is a constrained critical point of \eqref{eq:RINP}, and its $m$th
fixed-sector eigenfunction has $T_*$ as its $i$th node.
\end{theorem}

\begin{proof}
The forward implication follows from Theorem~\ref{thm:critical-system},
Lemma~\ref{lem:regular-shooting}, the nodal count, and the mass balance
\eqref{eq:mass-balance-main}.  For the converse, put $T=T_*$,
$\alpha_-=|U_-'(T)|$, and $\alpha_+=|U_+'(T)|$, and define
\[
 E(r)=\begin{cases}
 U_-(r)/\alpha_-,&0<r<T,\\
 U_+(r)/\alpha_+,&T<r<R.
 \end{cases}
\]
Then $E$ and its derivative match at $T$, so $E$ is a global Friedrichs weak
solution of
\[
 -E''-\frac{d-1}{r}E'+\frac{\gamma}{r^2}E+qE=\lambda E.
\]
The zero counts identify it as an $m$th fixed-sector eigenfunction.  If
$I=M_-=M_+$, then
$A=I/\alpha_-^2$, $B=I/\alpha_+^2$, and hence
$\alpha_-^2/\alpha_+^2=B/A$.
Put
\[
 M:=\int_0^R E^2\dd\mu_d=A+B.
\]
With $\Gamma=T^{d-1}E'(T)^2$, choose $\kappa$ by
\[
 |\kappa|=\frac{\alpha_-^2M\Gamma}{B}
 =\frac{\alpha_+^2M\Gamma}{A},
 \qquad \sgn\kappa=\varepsilon.
\]
Formula \eqref{eq:shooting-potential} then gives
$\varphi_p(q-q_0)=\kappa g_{i,m}(\cdot;q)$ through
\eqref{eq:unnormalized-kernel}.
Moreover, the two shooting profiles are continuous and bounded on their
respective compact radial intervals, and
\[
 |q-q_0|^p=|U_-|^{2p'}\ind_{(0,T)}
             +|U_+|^{2p'}\ind_{(T,R)}.
\]
Hence $q-q_0\in L_d^p$.
Thus the constrained Euler--Lagrange equation holds.
\end{proof}

\begin{remark}[Numerical reconstruction]\label{rem:numerics}
System \eqref{eq:shooting-system} can be treated by Newton or continuation
methods using the leading coefficient $a$, the terminal derivative $b$, and
the eigenvalue $\lambda$.  Solving it finds constrained critical points;
a second-variation or comparison argument is still required to identify
global minimizers.
\end{remark}

\subsection{Global uniqueness in the inward
second-radial-mode configuration}
\label{sec:global-uniqueness-special}

Only in this subsection do we specialize the fixed sector to the radial case
$\ell=0$.  We assume
\begin{equation}\label{eq:global-uniqueness-assumptions}
 d\ge2,\qquad q_0\equiv c,\qquad m=2,\qquad i=1,\qquad
 p>\frac{d+2}{2},
\end{equation}
and put
\begin{equation}\label{eq:global-sigma}
 \sigma:=2p',\qquad 2<\sigma<2+\frac4d.
\end{equation}
Thus $\sigma$ is Sobolev subcritical and, more specifically, $L^2$-mass
subcritical.
We record the proof mechanism before entering the details.  The equivalence
\[
 p>\frac{d+2}{2}
 \quad\Longleftrightarrow\quad
 \sigma=2p'<2+\frac4d
\]
is the only extra exponent restriction used in the global argument.  First,
domain monotonicity selects the critical sign and confines the shifted
eigenvalue $\mu$ to
$(\Lambda_R(T),\Lambda_L(T))$.  Second, the optimizer is represented by the
unique positive focusing branch on $B_T$ and the unique positive logistic
branch on $A_{T,R}$.  Third, the logistic mass is strictly increasing, whereas
the focusing mass is strictly decreasing.  The mass-subcritical inequality
enters precisely in the focusing calculation: after differentiating the
Pohozaev identity, it gives
\[
 \frac d2-\frac2{\sigma-2}<0.
\]
This strict sign rules out a zero of $M_L'(\mu)$; the vanishing of $M_L$ at
$\Lambda_L(T)$ then fixes $M_L'(\mu)<0$.  Finally, the endpoint limits and
opposite monotonicities give exactly one mass-balance parameter, which yields
the unique global minimizer.
Let
\begin{equation}\label{eq:global-reference-node}
 T_0:=T_{1,2}^{(0)}(c)
 =R\frac{j_{(d-2)/2,1}}{j_{(d-2)/2,2}}.
\end{equation}
For $T\in(0,R)$, let $\Lambda_L(T)$ be the first Dirichlet eigenvalue of
$-\Delta$ on the ball $B_T$, and let $\Lambda_R(T)$ be the first Dirichlet
eigenvalue on the annulus
\[
 A_{T,R}:=\{x\in\mathbb R^d:T<|x|<R\}.
\]
The restrictions of the second radial Dirichlet eigenfunction on $B_R$ to
$B_{T_0}$ and $A_{T_0,R}$ are positive first eigenfunctions.  Consequently,
\begin{equation}\label{eq:crossing-first-eigenvalues}
 \Lambda_L(T_0)=\Lambda_R(T_0)
 =\frac{j_{(d-2)/2,2}^2}{R^2}.
\end{equation}
Strict domain monotonicity gives, for $0<T<T_0$,
\begin{equation}\label{eq:strict-domain-order}
 \Lambda_R(T)<\Lambda_R(T_0)=\Lambda_L(T_0)<\Lambda_L(T).
\end{equation}

\begin{lemma}
\label{lem:global-sign-selection}
Assume \eqref{eq:global-uniqueness-assumptions} and let $0<T<T_0$.  Every
minimizer satisfying $T_{1,2}^{(0)}(\widehat q)=T$ has
\begin{equation}\label{eq:inward-critical-sign}
 \varepsilon=-1
\end{equation}
in Theorem~\ref{thm:critical-system}.  If
$\mu:=\lambda_{0,2}(\widehat q)-c$, then
\begin{equation}\label{eq:mu-gap}
 \Lambda_R(T)<\mu<\Lambda_L(T).
\end{equation}
\end{lemma}

\begin{proof}
The restriction of the second radial eigenfunction to either side of its
unique node has a fixed sign and, after an independent sign change, is a
positive first Dirichlet eigenfunction of the corresponding restricted
Schr\"odinger operator.  In particular, the two pieces of $U$ have no
interior zero and are nonzero almost everywhere on their domains.  Suppose
that $\varepsilon=1$.  Formula \eqref{eq:q-reconstruction-U} gives
\[
 \widehat q>c\quad\hbox{a.e. on }B_T,
 \qquad
 \widehat q<c\quad\hbox{a.e. on }A_{T,R}.
\]
The inequalities are strict on sets of positive measure.  Strict
monotonicity of the first eigenvalue with respect to the potential
would then imply
\[
 \lambda_{0,2}(\widehat q)>c+\Lambda_L(T),
 \qquad
 \lambda_{0,2}(\widehat q)<c+\Lambda_R(T),
\]
contradicting \eqref{eq:strict-domain-order}.  Hence $\varepsilon=-1$.
Reversing the potential inequalities on the two domains gives
\[
 \lambda_{0,2}(\widehat q)<c+\Lambda_L(T),
 \qquad
 \lambda_{0,2}(\widehat q)>c+\Lambda_R(T),
\]
which is \eqref{eq:mu-gap}.
\end{proof}

We next prove the strict mass monotonicity that replaces the periodic
quadrature available in the one-dimensional problem.  All solutions below
are positive and radial.

The sign selection above is global: it applies to every global minimizer
satisfying the nodal constraint, rather than only to critical points close to
the reference potential.

\begin{lemma}
\label{lem:opposite-mass-monotonicity}
Fix $T\in(0,T_0)$ and let $\sigma$ satisfy \eqref{eq:global-sigma}.  For every \(\mu\in(\Lambda_R(T),\Lambda_L(T))\)
there are unique positive solutions
\begin{align}
 -\Delta u_\mu&=\mu u_\mu+u_\mu^{\sigma-1}
 &&\text{in }B_T, &u_\mu&=0&&\text{on }\partial B_T,
 \label{eq:focusing-ball-branch}\\
 -\Delta v_\mu+v_\mu^{\sigma-1}&=\mu v_\mu
 &&\text{in }A_{T,R}, &v_\mu&=0&&\text{on }\partial A_{T,R}.
 \label{eq:logistic-annulus-branch}
\end{align}
For every fixed $r_*>d$, the two branches are $C^1$ with values in
\(W_{\rm rad}^{2,r_*}\cap W_0^{1,r_*}\)
on their parameter intervals, and hence their boundary slopes depend $C^1$
on $\mu$.  Their weighted radial masses
\begin{equation}\label{eq:one-domain-masses}
 M_L(\mu):=\int_0^T u_\mu(r)^2\,\dd\mu_d(r),
 \qquad
 M_R(\mu):=\int_T^R v_\mu(r)^2\,\dd\mu_d(r)
\end{equation}
satisfy
\begin{equation}\label{eq:opposite-mass-derivatives}
 M_L'(\mu)<0,\qquad M_R'(\mu)>0.
\end{equation}
Moreover,
\begin{equation}\label{eq:mass-endpoint-limits}
 \lim_{\mu\uparrow\Lambda_L(T)}M_L(\mu)=0,
 \qquad
 \lim_{\mu\downarrow\Lambda_R(T)}M_R(\mu)=0.
\end{equation}
The focusing solution at $\mu=\Lambda_R(T)$ and the logistic solution at
$\mu=\Lambda_L(T)$ are well defined and positive; in particular,
\begin{equation}\label{eq:opposite-endpoint-positive-masses}
 M_L(\Lambda_R(T))>0,
 \qquad
 M_R(\Lambda_L(T))>0.
\end{equation}
\end{lemma}

\begin{proof}
We begin with \eqref{eq:focusing-ball-branch}.  Since
$0<\mu<\Lambda_L(T)$ and $\sigma$ is Sobolev subcritical, the functional
\[
 \mathcal E_\mu(w)
 :=\frac12\int_{B_T}(|\nabla w|^2-\mu w^2)\,\dd x
   -\frac1\sigma\int_{B_T}|w|^\sigma\,\dd x
\]
attains its minimum on the Nehari manifold
\[
 \mathcal N_\mu
 :=\left\{w\in H_0^1(B_T)\setminus\{0\}:
 \int_{B_T}(|\nabla w|^2-\mu w^2)\,\dd x
 =\int_{B_T}|w|^\sigma\,\dd x\right\}.
\]
After replacing the minimizer by its absolute value, the strong maximum
principle gives a positive solution.  The moving-plane theorem
\cite{GidasNiNirenberg1979} makes every positive solution radial and strictly
decreasing in the radial variable.  To verify the precise range of the ball
uniqueness theorem, define
\[
 w(y):=T^{2/(\sigma-2)}u_\mu(Ty),\qquad y\in B_1.
\]
Then
\[
 -\Delta w=T^2\mu\,w+w^{\sigma-1}\quad\hbox{in }B_1,
 \qquad w|_{\partial B_1}=0,
\]
and
\[
 0<T^2\mu<T^2\Lambda_L(T)=\lambda_1(B_1),
 \qquad
 1<\sigma-1<1+\frac4d.
\]

Since $2<\sigma<2+4/d$, the exponent $\sigma-1$ is Sobolev subcritical
for every $d\ge2$; in dimension $d=2$ this uses
$H_0^1(B_1)\hookrightarrow L^q(B_1)$ for every finite $q$.  Writing
$\lambda=-T^2\mu\in(-\lambda_1(B_1),0)$, the rescaled equation becomes
\[
 -\Delta w+\lambda w=w^{\sigma-1}\quad\hbox{in }B_1,
 \qquad w|_{\partial B_1}=0.
\]
In the notation of
\cite[Theorem~5.1]{NorisTavaresVerzini2014}, the linear parameter is
$\lambda_{\mathrm{NTV}}=-T^2\mu\in(-\lambda_1(B_1),0)$, the nonlinear
coefficient is positive, and the power is
$p_{\mathrm{NTV}}=\sigma-1\in(1,1+4/d)$.  Thus all of its hypotheses are
satisfied.  The cited theorem gives both uniqueness of the positive solution
and triviality of the kernel of the full Dirichlet linearized operator, not
merely of its radial restriction.
Hence there is a unique positive solution $u_\mu$ for every $d\ge2$ in the
present exponent range.
Since the exponent is Sobolev subcritical, standard elliptic bootstrap gives
\[
 u_\mu\in L^\infty(B_T)\cap W^{2,r}(B_T)
 \quad\hbox{for every finite }r>1.
\]
We record the spectral information needed for parameter differentiation.
Define the Nehari constraint
\[
 \mathcal G_\mu(w)
 :=\int_{B_T}(|\nabla w|^2-\mu w^2)\,\dd x
   -\int_{B_T}|w|^\sigma\,\dd x.
\]
At $u_\mu$,
\[
 \mathcal G_\mu'(u_\mu)[u_\mu]
 =(2-\sigma)\int_{B_T}u_\mu^\sigma\,\dd x\ne0,
\]
so the tangent space to $\mathcal N_\mu$ has codimension one.  The second
variation is
\[
 Q_\mu(\phi)
 :=\int_{B_T}\left(|\nabla\phi|^2-\mu\phi^2
 -(\sigma-1)u_\mu^{\sigma-2}\phi^2\right)\dd x.
\]
The constrained minimality gives $Q_\mu\ge0$ on the tangent space, while
\[
 Q_\mu(u_\mu)=-(\sigma-2)\int_{B_T}u_\mu^\sigma\,\dd x<0.
\]
Hence the Morse index of $u_\mu$ is exactly one.  The full Dirichlet
nondegeneracy just recorded
shows that the linearized operator has trivial kernel, and
therefore its restriction to the radial subspace has trivial kernel as well.
For every fixed $r_*>d$, the Dirichlet
linearized operator is Fredholm of index zero from
$W_{\rm rad}^{2,r_*}(B_T)\cap W_0^{1,r_*}(B_T)$ to
$L_{\rm rad}^{r_*}(B_T)$; the trivial kernel therefore makes it an
isomorphism.  Applying the implicit function theorem to the $C^1$ map
\[
 (\mu,w)\longmapsto-\Delta w-\mu w-|w|^{\sigma-2}w
\]
and using uniqueness patches the local branches into a $C^1$ branch on
$(0,\Lambda_L(T))$.  Since
$W^{2,r_*}\hookrightarrow C^{1,\alpha}$ for some $\alpha>0$, differentiating
the Pohozaev identity, including its boundary term, is justified.

Set $u=u_\mu$ and $z=\partial_\mu u_\mu$.  Then
\begin{equation}\label{eq:focusing-linearized-equation}
 L_\mu z=u,
 \qquad
 L_\mu:=-\Delta-\mu-(\sigma-1)u^{\sigma-2}.
\end{equation}
Use radial integrals and write
\[
 I:=\int_0^Tuz\,\dd\mu_d,
 \qquad
 J:=\int_0^Tu^{\sigma-1}z\,\dd\mu_d.
\]
Since $L_\mu u=-(\sigma-2)u^{\sigma-1}$, radial self-adjointness gives
\begin{equation}\label{eq:J-identity}
 J=-\frac1{\sigma-2}\int_0^Tu^2\,\dd\mu_d<0.
\end{equation}
We claim that $I\ne0$.  Suppose instead that $I=0$.  Dividing the standard
Pohozaev identity by $|\mathbb S^{d-1}|$ gives
\begin{equation}\label{eq:focusing-pohozaev}
 \frac{d-2}{2}\int_0^T u'(r)^2\,\dd\mu_d
 +\frac{T^d}{2}u'(T)^2
 =\frac{d\mu}{2}\int_0^T u^2\,\dd\mu_d
 +\frac d\sigma\int_0^T u^\sigma\,\dd\mu_d.
\end{equation}
Differentiating \eqref{eq:focusing-pohozaev}, using the equation tested
against $z$, and then using $I=0$ and \eqref{eq:J-identity}, yields
\begin{equation}\label{eq:pohozaev-boundary-sign}
 T^d u'(T)z'(T)
 =\left(\frac d2-\frac2{\sigma-2}\right)
  \int_0^T u^2\,\dd\mu_d<0.
\end{equation}
The strict inequality follows precisely from
$\sigma<2+4/d$, or equivalently $p>(d+2)/2$.  It is used here to exclude the
possibility $M_L'(\mu)=2I=0$; the endpoint behavior below then determines the
strict negative sign of $M_L'$. Since
$u'(T)<0$, we have $z'(T)>0$, and therefore $z<0$ immediately to the left of
$T$.

The negative set of $z$ has at most one radial component.  Indeed, suppose
that $D_1,D_2$ are two disjoint negative radial nodal regions.  Define
\[
 \phi_j:=-z\mathbf1_{D_j},\qquad j=1,2.
\]
If a component has an endpoint at the origin, radial regularity of $z$ gives
the required finite-energy behavior there; at every interior endpoint the
trace of $z$ is zero.  Thus the zero extensions satisfy
$\phi_j\in H_{0,\rm rad}^1(B_T)$.  On $D_j$ we have
$L_\mu\phi_j=-u$, and therefore
\[
 Q_\mu(\phi_j)
 =-\int_{D_j}u\phi_j\,\dd x<0,
 \qquad j=1,2.
\]
Because the supports are disjoint, the cross term vanishes and
\[
 Q_\mu(a\phi_1+b\phi_2)
 =a^2Q_\mu(\phi_1)+b^2Q_\mu(\phi_2)<0
\]
for every $(a,b)\ne(0,0)$.  Hence $Q_\mu$ is negative definite on the
two-dimensional space $\operatorname{span}\{\phi_1,\phi_2\}$, contradicting
Morse index one.  Moreover, any interior zero $r_*$ of $z$ with
$z'(r_*)=0$ satisfies $z''(r_*)=-u(r_*)<0$ by
\eqref{eq:focusing-linearized-equation}; thus the zeros relevant to the nodal
decomposition are isolated.  Since $I=0$, $u>0$, and $z<0$ near $T$, the
function $z$ is positive somewhere.  Consequently, the unique negative
component is $(r_0,T)$ for some $r_0\in(0,T)$, and, apart from isolated
zeros,
\[
 z\ge0\quad\hbox{on }(0,r_0),
 \qquad
 z<0\quad\hbox{on }(r_0,T).
\]
The positive solution $u$ is strictly decreasing.  Using $I=0$, we obtain
\[
 J=\int_0^T
 \bigl(u^{\sigma-2}-u(r_0)^{\sigma-2}\bigr)uz\,\dd\mu_d>0,
\]
contradicting \eqref{eq:J-identity}.  Hence $I\ne0$.  Since \(M_L'(\mu)=2I\),
the derivative never vanishes.  It is continuous and the parameter interval
is connected, so its sign is constant.

To determine the sign and prove the first endpoint limit, let $\phi_1$ be a
positive first eigenfunction on $B_T$.  Choosing $t_\mu>0$ so that
$t_\mu\phi_1\in\mathcal N_\mu$ gives
\[
 t_\mu^{\sigma-2}
 =(\Lambda_L(T)-\mu)
 \frac{\int_{B_T}\phi_1^2\,\dd x}{\int_{B_T}\phi_1^\sigma\,\dd x}.
\]
Since $u_\mu$ is the Nehari minimizer,
\[
 \int_{B_T}u_\mu^\sigma\,\dd x
 \le C(\Lambda_L(T)-\mu)^{\sigma/(\sigma-2)}.
\]
The Nehari identity and the Poincar\'e inequality then imply
\[
 \int_{B_T}|\nabla u_\mu|^2\,\dd x
 \le C(\Lambda_L(T)-\mu)^{2/(\sigma-2)}\longrightarrow0.
\]
Thus $M_L(\mu)\to0$ as $\mu\uparrow\Lambda_L(T)$.  If $M_L'$ were positive,
then for fixed $\mu<\mu_2<\Lambda_L(T)$ one would have
$0<M_L(\mu)<M_L(\mu_2)$; letting $\mu_2\uparrow\Lambda_L(T)$ gives a
contradiction.  Hence $M_L'<0$.

We turn to \eqref{eq:logistic-annulus-branch}.  The functional
\[
 \mathcal J_\mu(w)
 :=\frac12\int_{A_{T,R}}|\nabla w|^2\,\dd x
  +\frac1\sigma\int_{A_{T,R}}|w|^\sigma\,\dd x
  -\frac\mu2\int_{A_{T,R}}w^2\,\dd x
\]
is coercive and has negative values on small multiples of the first annular
eigenfunction because $\mu>\Lambda_R(T)$.  Hence it has a nonzero positive
minimizer.  Uniqueness follows from the Brezis--Oswald principle
\cite{BrezisOswald1986}, since
\[
 \frac{\mu s-s^{\sigma-1}}s=\mu-s^{\sigma-2}
\]
is strictly decreasing for $s>0$.  Rotational invariance and uniqueness make
this solution radial.  The maximum principle and subcritical elliptic
bootstrap give
\[
 v_\mu\in L^\infty(A_{T,R})\cap W^{2,r}(A_{T,R})
 \quad\hbox{for every finite }r>1.
\]

Let
\[
 \mathcal L_\mu:=-\Delta+(\sigma-1)v_\mu^{\sigma-2}-\mu.
\]
Because
$\mathcal L_\mu v_\mu=(\sigma-2)v_\mu^{\sigma-1}>0$, pairing this identity
with the positive principal eigenfunction of $\mathcal L_\mu$ shows that its
principal eigenvalue is positive.  Hence the kernel is trivial.  The
Dirichlet operator is Fredholm of index zero from
$W^{2,r_*}\cap W_0^{1,r_*}$ to $L^{r_*}$, and is therefore an isomorphism.  Applying the implicit function
theorem to $(\mu,w)\mapsto-\Delta w+|w|^{\sigma-2}w-\mu w$ gives a
$C^1$ branch in $W^{2,r_*}$.  If
$w=\partial_\mu v_\mu$, then
$\mathcal L_\mu w=v_\mu>0$; the maximum principle yields $w>0$, and hence
$M_R'(\mu)>0$.

At a maximum point of $v_\mu$, the equation gives
$\|v_\mu\|_\infty^{\sigma-2}\le\mu$.  Hence the branch is uniformly bounded
as $\mu\downarrow\Lambda_R(T)$.  Let
$\mu_n\downarrow\Lambda_R(T)$.  The equation and the uniform $L^\infty$
bound give uniform $W^{2,r}$ bounds for every finite $r$; after passing to a
subsequence,
\[
 v_{\mu_n}\longrightarrow v_*
 \quad\hbox{in }C^1(\overline{A_{T,R}})
\]
for some nonnegative Dirichlet solution at $\mu=\Lambda_R(T)$.  Testing the
equation against the positive first annular eigenfunction $\phi_R$ gives
\[
 (\mu_n-\Lambda_R(T))\int_{A_{T,R}}v_{\mu_n}\phi_R\,\dd x
 =\int_{A_{T,R}}v_{\mu_n}^{\sigma-1}\phi_R\,\dd x.
\]
Passing to the limit yields
\[
 \int_{A_{T,R}}v_*^{\sigma-1}\phi_R\,\dd x=0.
\]
Since $v_*\ge0$ and $\phi_R>0$, we have $v_*\equiv0$.  Thus every sequence
approaching $\Lambda_R(T)$ has a subsequence converging to zero, and therefore
\[
 M_R(\mu)\longrightarrow0
 \quad\hbox{as }\mu\downarrow\Lambda_R(T).
\]
This proves the second endpoint limit.

Finally, $\Lambda_R(T)$ lies strictly below $\Lambda_L(T)$, so the focusing
existence, uniqueness, regularity, and nondegeneracy arguments also apply at
$\mu=\Lambda_R(T)$.  The implicit function theorem therefore extends the
focusing branch continuously to this endpoint.  Similarly, at
$\mu=\Lambda_L(T)$ the logistic solution exists uniquely and its linearized
operator has positive principal eigenvalue, so the logistic branch extends
continuously there.  The endpoint solutions are positive.  Consequently the
two mass functions extend continuously to the opposite endpoints and satisfy
\eqref{eq:opposite-endpoint-positive-masses}.
\end{proof}

\begin{theorem}\label{thm:global-uniqueness-inward}
Assume \eqref{eq:global-uniqueness-assumptions}.  Then, for every
\begin{equation}\label{eq:inward-node-range}
 0<T_*<T_0,
\end{equation}
problem \eqref{eq:RINP}, interpreted in the radial sector $\ell=0$ with
$m=2$ and $i=1$, has a unique global minimizer.  More precisely, put
$T=T_*$ and define
\begin{equation}\label{eq:mass-difference-function}
 \Phi_T(\mu):=M_L(\mu)-M_R(\mu),
 \qquad \Lambda_R(T)<\mu<\Lambda_L(T).
\end{equation}
There is a unique $\mu_T$ satisfying $\Phi_T(\mu_T)=0$, and the unique
minimizer is represented, for almost every $r\in(0,R)$, by
\begin{equation}\label{eq:global-unique-potential}
 \widehat q_T(r)=
 \begin{cases}
  c-u_{\mu_T}(r)^{\sigma-2},&0<r<T,\\
  c+v_{\mu_T}(r)^{\sigma-2},&T<r<R.
 \end{cases}
\end{equation}
Here $u_{\mu_T}$ and $v_{\mu_T}$ are the positive radial solutions of
\eqref{eq:focusing-ball-branch} and \eqref{eq:logistic-annulus-branch}.
Moreover, $\widehat q_T-c\in L_d^p$ because
\begin{equation}\label{eq:global-potential-integrability}
 (\sigma-2)p=\sigma
\end{equation}
and $u_{\mu_T},v_{\mu_T}\in L^\sigma$ on their respective domains.
\end{theorem}

\begin{proof}
By Lemma~\ref{lem:opposite-mass-monotonicity}, $\Phi_T$ is continuous and
strictly decreasing.  Equations \eqref{eq:mass-endpoint-limits} and
\eqref{eq:opposite-endpoint-positive-masses} give
\[
 \lim_{\mu\downarrow\Lambda_R(T)}\Phi_T(\mu)
 =M_L(\Lambda_R(T))>0,
\]
whereas
\[
 \lim_{\mu\uparrow\Lambda_L(T)}\Phi_T(\mu)
 =-M_R(\Lambda_L(T))<0.
\]
Thus $\Phi_T$ has exactly one zero $\mu_T$.

Existence of a global minimizer follows from Theorem~\ref{thm:existence}.
Let $\widehat q$ be any minimizer.  Since $T<T_0$, it is nontrivial.
Lemma~\ref{lem:global-sign-selection} gives $\varepsilon=-1$ and
\eqref{eq:mu-gap}.  Since $m=2$ and $i=1$, the two pieces of the scaled
function in Theorem~\ref{thm:critical-system} have no interior zeros.  After
independent sign changes, their absolute values are the positive solutions
of \eqref{eq:focusing-ball-branch} and
\eqref{eq:logistic-annulus-branch}, with
$\mu=\lambda_{0,2}(\widehat q)-c$.  The weighted balance law
\eqref{eq:mass-balance-main} is exactly
\[
 M_L(\mu)=M_R(\mu).
\]
Uniqueness of the two one-domain positive solutions and uniqueness of the
zero of $\Phi_T$ imply $\mu=\mu_T$ and force
\eqref{eq:global-unique-potential} almost everywhere.  Thus every global
minimizer represents the same element of $L_d^p$.
\end{proof}

\begin{remark}\label{rem:outward-global-open}
The theorem is deliberately one-sided. If $T_*>T_0$, the focusing
equation occurs on an annulus rather than on a ball.  The radial ball
uniqueness and Morse-index argument above no longer provide the required
global annular mass monotonicity, and the present method does not exclude
multiple positive annular branches.  The outward problem
therefore remains open.  After this subsection, the fixed-sector notation
again refers to an arbitrary $\ell\ge0$.
\end{remark}

\section{Local reconstruction from several nodes of one fixed mode}\label{sec:multi-node}

In this section we prescribe finitely many nodes of the same fixed-sector eigenfunction.  Since all corresponding gradients contain the common factor $E_m^2$, their linear independence can be proved directly.  We fix
\begin{equation}\label{eq:multi-index-set}
 m\ge2,\qquad
 \mathcal I=\{i_1<i_2<\cdots<i_N\}\subset\{1,\ldots,m-1\}.
\end{equation}
Define the vector-valued nodal map

\[
 \mathbf T(q)=\mathbf T_{\mathcal I,m}(q)
 :=\big(T_{i_1,m}(q),\ldots,T_{i_N,m}(q)\big)^{\mathsf T}
 \in\R^N.
\]

For a compatible target vector
\[
 \boldsymbol\tau=(\tau_1,\ldots,\tau_N)^{\mathsf T},
 \qquad 0<\tau_1<\cdots<\tau_N<R,
\]
we consider
\begin{equation}\label{eq:multi-node-problem}
 \min\left\{\|q-q_0\|_{L_d^p}:
 \mathbf T(q)=\boldsymbol\tau,\ q\in L_d^p\right\}.
\end{equation}
By the finite-node realizability result established above,
a target vector is compatible precisely when its prescribed
radii respect the ordering of the selected nodal indices.

\subsection{Linear independence and submersion}

Fix $q\in L_d^p$ and let $E=E_m(\cdot;q)$ be normalized by \eqref{eq:normalization}.  For $1\le\alpha\le N$, put
\begin{equation}\label{eq:multi-node-quantities}
 T_\alpha:=T_{i_\alpha,m}(q),\quad
 A_\alpha:=\int_0^{T_\alpha}E^2\dd\mu_d,\quad
 B_\alpha:=\int_{T_\alpha}^R E^2\dd\mu_d,\quad
 \Gamma_\alpha:=T_\alpha^{d-1}E'(T_\alpha)^2,
\end{equation}
and denote the corresponding gradient by

\[
 g_\alpha(r;q)
 =E(r)^2H_\alpha(r;q),
 \qquad
 H_\alpha(r;q)
 =\frac{B_\alpha}{\Gamma_\alpha}\ind_{(0,T_\alpha)}(r)
 -\frac{A_\alpha}{\Gamma_\alpha}\ind_{(T_\alpha,R)}(r).
\]

\begin{theorem}\label{thm:multi-gradient-independence}
Under \eqref{eq:main-p-assumption}, the family
\(\{g_1(\cdot;q),\ldots,g_N(\cdot;q)\}\)
is linearly independent in $L_d^{p'}$ for every $q\in L_d^p$.
\end{theorem}

\begin{proof}
Suppose that
\[
 \sum_{\alpha=1}^Nc_\alpha g_\alpha=0
 \qquad\text{a.e. on }(0,R).
\]
The eigenfunction $E$ has only finitely many zeros, hence $E^2>0$ almost everywhere.  Dividing by $E^2$ gives
\begin{equation}\label{eq:step-combination-zero}
 \sum_{\alpha=1}^Nc_\alpha H_\alpha=0
 \qquad\text{a.e. on }(0,R).
\end{equation}
The left-hand side is a step function whose possible jumps are the distinct points $T_1,\ldots,T_N$.  At $T_\beta$, only $H_\beta$ jumps, and its jump is
\[
 H_\beta(T_\beta^+)-H_\beta(T_\beta^-)
 =-\frac{A_\beta+B_\beta}{\Gamma_\beta}
 =-\frac1{\Gamma_\beta}.
\]
Since the step function in \eqref{eq:step-combination-zero} vanishes almost everywhere, its constant value on each complementary interval is zero, and therefore every jump is zero.  Thus
\( -\frac{c_\beta}{\Gamma_\beta}=0\).
As $\Gamma_\beta>0$, we have $c_\beta=0$ for every $\beta$.
\end{proof}

Let

\[
 G(q):=D\mathbf T(q):L_d^p\longrightarrow\R^N,
 \qquad
 G(q)h=\left(\int_0^Rg_\alpha(r;q)h(r)\dd\mu_d(r)\right)_{\alpha=1}^N.
\]

Its adjoint in the duality between $L_d^p$ and $L_d^{p'}$ is

\[
 G(q)^*\boldsymbol\lambda
 =\sum_{\alpha=1}^N\lambda_\alpha g_\alpha(\cdot;q),
 \qquad \boldsymbol\lambda\in\R^N.
\]

\begin{theorem}\label{thm:multi-submersion}
The map $\mathbf T:L_d^p\to\R^N$ is continuously Fr\'echet differentiable and $G(q)=D\mathbf T(q)$ is surjective at every $q\in L_d^p$.  Consequently, every joint level set
\[
 \{\widetilde q:\mathbf T(\widetilde q)=\mathbf T(q)\}
\]
is, locally near $q$, a $C^1$ Banach submanifold of codimension $N$.
\end{theorem}

\begin{proof}
Componentwise continuous differentiability follows from Theorem~\ref{thm:node-calculus}.  If $G(q)$ were not onto, its range would be a proper subspace of $\R^N$, so there would exist $\boldsymbol\lambda\ne0$ such that
\[
 \boldsymbol\lambda\cdot G(q)h=0
 \qquad\text{for all }h\in L_d^p.
\]
Equivalently, $G(q)^*\boldsymbol\lambda=0$ in $L_d^{p'}$, contradicting Theorem~\ref{thm:multi-gradient-independence}.  The submersion theorem gives the manifold statement.
\end{proof}

\subsection{Global feasibility, existence, and the multi-node critical equation}

The gluing construction of Proposition~\ref{prop:global-feasibility} actually prescribes the complete nodal vector of one fixed-sector mode.

\begin{proposition}\label{prop:multi-global-feasibility}
Let \(0<\tau_1<\cdots<\tau_{m-1}<R\).
There exists a bounded radial piecewise constant potential $q_{\boldsymbol\tau}$ such that
\[
 T_{j,m}(q_{\boldsymbol\tau})=\tau_j,
 \qquad 1\le j\le m-1.
\]
Consequently, every compatible partial vector $(\tau_1,\ldots,\tau_N)$ associated with the indices in \eqref{eq:multi-index-set} is feasible.
\end{proposition}

\begin{proof}
Apply the construction in the proof of Proposition~\ref{prop:global-feasibility} to the partition
\[
 0=r_0<r_1=\tau_1<\cdots<r_{m-1}=\tau_{m-1}<r_m=R.
\]
The resulting global eigenfunction has precisely these $m-1$ simple zeros.  For a partial nodal vector, insert $i_1-1$ points before $\tau_1$, insert $i_{\alpha+1}-i_\alpha-1$ points between $\tau_\alpha$ and $\tau_{\alpha+1}$, and insert $m-1-i_N$ points after $\tau_N$.  Every nonempty open interval contains the required finite number of points.  The completed vector has $m-1$ ordered entries and places each prescribed radius in its required nodal position, so the full-vector construction realizes the selected components.
\end{proof}

\begin{theorem}\label{thm:multi-existence}
Assume \eqref{eq:main-p-assumption}.  For every compatible target vector $\boldsymbol\tau$, problem \eqref{eq:multi-node-problem} has at least one minimizer.
\end{theorem}

\begin{proof}

By Proposition~\ref{prop:multi-global-feasibility}, the constraint set
\[
\mathcal A_{\boldsymbol\tau}
:=
\left\{
q\in L_d^p:\mathbf T(q)=\boldsymbol\tau
\right\}
\]
is nonempty. We first show that $\mathcal A_{\boldsymbol\tau}$ is weakly sequentially closed in
$L_d^p$. Let $\{q_n\}\subset \mathcal A_{\boldsymbol\tau}$ and suppose that
\[
q_n\rightharpoonup q
\qquad\text{weakly in }L_d^p.
\]
Since $q_n\in\mathcal A_{\boldsymbol\tau}$, we have
\[
T_{i_\alpha,m}(q_n)=\tau_\alpha,
\qquad 1\le \alpha\le N.
\]
By Theorem~\ref{thm:node-calculus}, each nodal map $T_{i_\alpha,m}$ is weakly
sequentially continuous. Hence
\[
T_{i_\alpha,m}(q_n)
\longrightarrow
T_{i_\alpha,m}(q),
\qquad 1\le \alpha\le N.
\]
Passing to the limit in the preceding identities yields
\[
T_{i_\alpha,m}(q)=\tau_\alpha,
\qquad 1\le \alpha\le N,
\]
and therefore
\(\mathbf T(q)=\boldsymbol\tau.\)
Thus $q\in\mathcal A_{\boldsymbol\tau}$, proving that
$\mathcal A_{\boldsymbol\tau}$ is weakly sequentially closed.

Set
\[
d_{\boldsymbol\tau}
:=
\inf_{q\in\mathcal A_{\boldsymbol\tau}}
\|q-q_0\|_{L_d^p}.
\]
Since $\mathcal A_{\boldsymbol\tau}\neq\varnothing$, we have $d_{\boldsymbol\tau}<\infty$.
Let $\{q_n\}\subset\mathcal A_{\boldsymbol\tau}$ be a minimizing sequence, that is,
\[
\|q_n-q_0\|_{L_d^p}
\longrightarrow d_{\boldsymbol\tau}
\qquad\text{as }n\to\infty.
\]
In particular, $\{\|q_n-q_0\|_{L_d^p}\}$ is bounded. By the triangle
inequality,
\[
\|q_n\|_{L_d^p}
\le
\|q_n-q_0\|_{L_d^p}
+
\|q_0\|_{L_d^p},
\]
so $\{q_n\}$ is bounded in $L_d^p$.
Hence, after passing to a subsequence if necessary, there
exists $q_*\in L_d^p$ such that
\[
q_n\rightharpoonup q_*
\qquad\text{weakly in }L_d^p.
\]
Since $\mathcal A_{\boldsymbol\tau}$ is weakly sequentially closed, it follows that \(q_*\in\mathcal A_{\boldsymbol\tau}\).
Finally, the norm in $L_d^p$ is weakly lower semicontinuous. Therefore,
\(
\|q_*-q_0\|_{L_d^p}\le\liminf_{n\to\infty}\|q_n-q_0\|_{L_d^p}=d_{\boldsymbol\tau}.
\)
On the other hand, since $q_*\in\mathcal A_{\boldsymbol\tau}$, the definition of
$d_{\boldsymbol\tau}$ gives \(d_{\boldsymbol\tau}\le\|q_*-q_0\|_{L_d^p}\).
Consequently,
\(\|q_*-q_0\|_{L_d^p}=d_{\boldsymbol\tau}\),
and hence $q_*$ is a minimizer of problem~\eqref{eq:multi-node-problem}.
\end{proof}

We next characterize the variational structure of nontrivial minimizers under several nodal constraints.
\begin{theorem}\label{thm:multi-critical}
Let $\widehat q$ be a minimizer of \eqref{eq:multi-node-problem} with $\widehat q\ne q_0$.  Let $E=E_m(\cdot;\widehat q)$ be normalized in $L_d^2$, and define $T_\alpha,A_\alpha,B_\alpha,\Gamma_\alpha$ by \eqref{eq:multi-node-quantities}.  Then there exists a nonzero multiplier vector $\boldsymbol\kappa=(\kappa_1,\ldots,\kappa_N)^{\mathsf T}\in\R^N$ such that
\begin{equation}\label{eq:multi-Lagrange}
 \varphi_p(\widehat q-q_0)
 =G(\widehat q)^*\boldsymbol\kappa
 =\sum_{\alpha=1}^N\kappa_\alpha g_\alpha(\cdot;\widehat q).
\end{equation}
If

\[
 H_{\boldsymbol\kappa}(r)
 :=\sum_{\alpha=1}^N\kappa_\alpha
 \left[
 \frac{B_\alpha}{\Gamma_\alpha}\ind_{(0,T_\alpha)}(r)
 -\frac{A_\alpha}{\Gamma_\alpha}\ind_{(T_\alpha,R)}(r)
 \right],
\]

then
\begin{equation}\label{eq:multi-q-reconstruction}
 \widehat q-q_0
 =|H_{\boldsymbol\kappa}|^{p'-2}H_{\boldsymbol\kappa}
 |E|^{2p'-2},
\end{equation}
and $E$ satisfies, weakly on each component of $(0,R)\setminus\{T_1,\ldots,T_N\}$ and hence almost everywhere there,
\begin{equation}\label{eq:multi-semilinear-E}
 -E''-\frac{d-1}{r}E'+\frac{\gamma}{r^2}E+q_0E
 +|H_{\boldsymbol\kappa}|^{p'-2}H_{\boldsymbol\kappa}
 |E|^{2p'-2}E
 =\lambda_m(\widehat q)E.
\end{equation}
Moreover,
\begin{align}
 \int_0^R H_{\boldsymbol\kappa}E^2\dd\mu_d&=0,
 \label{eq:multi-balance}\\
 \|\widehat q-q_0\|_{L_d^p}^p
 &=\int_0^R|H_{\boldsymbol\kappa}|^{p'}|E|^{2p'}\dd\mu_d.
 \label{eq:multi-norm}
\end{align}
\end{theorem}

\begin{proof}
The derivative $G(\widehat q)$ is onto by Theorem~\ref{thm:multi-submersion}.
The Banach-space Lagrange multiplier theorem applied to $p^{-1}\|q-q_0\|_{L_d^p}^p$ gives \eqref{eq:multi-Lagrange}.
We claim that $\boldsymbol\kappa\neq \boldsymbol 0$. Indeed, if $\boldsymbol\kappa=\boldsymbol 0$, then
\eqref{eq:multi-Lagrange} gives
\[
\varphi_p(\widehat q-q_0)=0
\qquad\text{a.e. on }(0,R).
\]
Since \(\varphi_p(s)=|s|^{p-2}s\)
vanishes if and only if $s=0$, it follows that
\[
\widehat q-q_0=0
\qquad\text{a.e. on }(0,R),
\]
and hence
\(
\widehat q=q_0
\,\text{in }L_d^p
\).
This contradicts the assumption $\widehat q\neq q_0$. Therefore, \(\boldsymbol\kappa\neq\boldsymbol 0\).
Since the right-hand side equals $H_{\boldsymbol\kappa}E^2$ and $\varphi_p^{-1}=\varphi_{p'}$, formula \eqref{eq:multi-q-reconstruction} follows.  Substitution into the eigenvalue equation gives \eqref{eq:multi-semilinear-E}.  Each nodal gradient has zero integral by \eqref{eq:translation-orthogonality}; hence \eqref{eq:multi-balance}.  Finally, $(p'-1)p=p'$, which yields \eqref{eq:multi-norm}.
\end{proof}

\begin{remark}\label{rem:multi-piecewise}
The function $H_{\boldsymbol\kappa}$ is constant on every component of
\((0,R)\setminus\{T_1,\ldots,T_N\}\).
If its value on such an interval is $c\ne0$ and $U=|c|^{1/2}E$, then \eqref{eq:multi-semilinear-E} becomes
\[
 -U''-\frac{d-1}{r}U'+\frac{\gamma}{r^2}U+q_0U
 +\sgn(c)\varphi_{2p'}(U)
 =\lambda_m(\widehat q)U.
\]
If $c=0$, the equation is linear on that interval.  Thus the two-piece one-node critical system is replaced by a finite multi-piece system whose signs and amplitudes are determined by the multiplier vector.
\end{remark}

\subsection{Local reconstruction as an application of the observation principle}
Fix $q_0\in L_d^p$ and put
\[
 \mathbf T_0:=\mathbf T(q_0),\qquad
 G_0:=D\mathbf T(q_0):L_d^p\to\mathbb R^N.
\]
For $\boldsymbol y\in\mathbb R^N$, define $\rho_0(\boldsymbol y)$ and
$\mathscr R_0(\boldsymbol y)$ by \eqref{eq:abstract-rho-R} with
$\mathcal O=\mathbf T$.

\begin{theorem}
\label{thm:multi-local-theory}
Assume \eqref{eq:main-p-assumption}.  Let
$\boldsymbol\delta\to0$ through vectors for which
$\mathbf T_0+\boldsymbol\delta$ remains compatible, and let
$\widehat q_{\boldsymbol\delta}$ be any minimizer satisfying
$\mathbf T(\widehat q_{\boldsymbol\delta})=\mathbf T_0+\boldsymbol\delta$.
Then
\begin{align}
 \widehat q_{\boldsymbol\delta}-q_0
 &=\mathscr R_0(\boldsymbol\delta)
   +o_{L_d^p}(|\boldsymbol\delta|),
 \label{eq:multi-local-reconstruction}\\
 \|\widehat q_{\boldsymbol\delta}-q_0\|_{L_d^p}
 &=\rho_0(\boldsymbol\delta)+o(|\boldsymbol\delta|).
 \label{eq:multi-local-distance}
\end{align}
If $d\in\{2,3\}$ and $p=2$, define
\[
 \mathbb G_0
 :=\bigl(\langle g_\alpha,g_\beta\rangle_{L_d^2}\bigr)_{\alpha,\beta=1}^N
 =G_0G_0^*.
\]
Then $\mathbb G_0$ is positive definite, the minimizer is unique for all
sufficiently small compatible $\boldsymbol\delta$, depends $C^1$ on the data,
and
\begin{align}
 \mathscr R_0(\boldsymbol\delta)
 &=G_0^*\mathbb G_0^{-1}\boldsymbol\delta,
 \label{eq:Hilbert-right-inverse}\\
 q_{\boldsymbol\delta}
 &=q_0+G_0^*\mathbb G_0^{-1}\boldsymbol\delta
   +O_{L_d^2}(|\boldsymbol\delta|^2),
 \label{eq:multi-Hilbert-second-order}\\
 \|q_{\boldsymbol\delta}-q_0\|_{L_d^2}^2
 &=\boldsymbol\delta^{\mathsf T}\mathbb G_0^{-1}\boldsymbol\delta
   +O(|\boldsymbol\delta|^3).
 \label{eq:Hilbert-reconstruction}
\end{align}
\end{theorem}
\begin{proof}
Theorem~\ref{thm:multi-submersion} gives the surjectivity of $G_0$, while
Theorem~\ref{thm:node-calculus} gives weak sequential continuity of every
component of $\mathbf T$.  Thus
Proposition~\ref{prop:finite-observation-principle} applies and yields
\eqref{eq:multi-local-reconstruction}--\eqref{eq:multi-local-distance}.
For $p=2$ and $d\in\{2,3\}$, Theorem~\ref{thm:multi-gradient-independence}
makes $\mathbb G_0$ positive definite, and
Proposition~\ref{prop:Hilbert-C2}, applied componentwise, makes $\mathbf T$
$C^2$.  The Hilbert conclusions now follow from
\eqref{eq:abstract-Hilbert-right-inverse}--\eqref{eq:abstract-Hilbert-cost}.
\end{proof}

\begin{corollary}\label{cor:multi-value-Hessian}
Define
\[
 \mathcal V(\boldsymbol\delta)
 :=\frac12\min\left\{
 \|q-q_0\|_{L_d^2}^2:
 \mathbf T(q)=\mathbf T_0+\boldsymbol\delta
 \right\}.
\]
Then
\(
 \mathcal V(\boldsymbol\delta)
 =\frac12\boldsymbol\delta^{\mathsf T}
 \mathbb G_0^{-1}\boldsymbol\delta
 +O(|\boldsymbol\delta|^3).
\)
In particular,
\(
 D\mathcal V(0)=0,
 \,
 D^2\mathcal V(0)=\mathbb G_0^{-1}.
\)
\end{corollary}

\begin{proof}
Let $(h(\boldsymbol\delta),\boldsymbol\kappa(\boldsymbol\delta))$ be the
$C^1$ KKT branch produced in
Theorem~\ref{thm:multi-local-theory}.  For
$\boldsymbol y\in\mathbb R^N$, differentiation of the constraint gives
\(G(q_0+h(\boldsymbol\delta))
 Dh(\boldsymbol\delta)[\boldsymbol y]=\boldsymbol y\).
Using the first KKT equation,
$h=G(q_0+h)^*\boldsymbol\kappa$, we obtain
\begin{align*}
 D\mathcal V(\boldsymbol\delta)[\boldsymbol y]
 &=\langle h(\boldsymbol\delta),
 Dh(\boldsymbol\delta)[\boldsymbol y]\rangle_{L_d^2}=\boldsymbol\kappa(\boldsymbol\delta)\cdot\boldsymbol y.
\end{align*}
Thus $\mathcal V$ is $C^2$ and
$D\mathcal V(\boldsymbol\delta)=\boldsymbol\kappa(\boldsymbol\delta)$.
The expansion obtained in the proof of
Theorem~\ref{thm:multi-local-theory} gives
\[
 \boldsymbol\kappa(\boldsymbol\delta)
 =\mathbb G_0^{-1}\boldsymbol\delta+O(|\boldsymbol\delta|^2).
\]
Integrating along the segment $t\boldsymbol\delta$, $0\le t\le1$, yields
\[
 \mathcal V(\boldsymbol\delta)
 =\int_0^1
 \boldsymbol\kappa(t\boldsymbol\delta)\cdot\boldsymbol\delta\,\dd t
 =\frac12\boldsymbol\delta^{\mathsf T}
 \mathbb G_0^{-1}\boldsymbol\delta+O(|\boldsymbol\delta|^3),
\]
and the derivative identities follow.
\end{proof}

\begin{remark}\label{rem:multi-condition-number}
The matrix $\mathbb G_0^{-1}$ is the local metric of reconstruction cost in nodal-data space.  In particular,
\[
 \frac{|\boldsymbol\delta|^2}{\lambda_{\max}(\mathbb G_0)}
 \le\boldsymbol\delta^{\mathsf T}\mathbb G_0^{-1}\boldsymbol\delta
 \le\frac{|\boldsymbol\delta|^2}{\lambda_{\min}(\mathbb G_0)}.
\]
A small value of $\lambda_{\min}(\mathbb G_0)$ signals a poorly identifiable combination of nodal displacements, while the condition number of $\mathbb G_0$ quantifies anisotropy of the local inverse problem.
\end{remark}

\section{Mixed eigenmodes, angular-momentum sectors, and spectral data}
\label{sec:mixed-data}

The preceding results hold for every fixed $\ell$.  We now restore the full notation to consider observations coming from different angular-momentum sectors and eigenmodes.  For $j\in\mathbb N_0$, write
\begin{equation}\label{eq:full-sector-parameters}
 \gamma_j:=j(j+d-2),\qquad \nu_j:=j+\frac{d-2}{2}.
\end{equation}

\subsection{Mixed eigenmodes and angular-momentum sectors}

Let
\begin{equation}\label{eq:mixed-index-family}
 \mathfrak a
 =\{(\ell_\alpha,m_\alpha,i_\alpha):1\leq\alpha\leq N\}
\end{equation}
be a family of nodal observations, and define
\begin{equation}\label{eq:mixed-nodal-map}
 \mathbf T_{\mathfrak a}(q)
 :=\left(
 T_{i_1,m_1}^{(\ell_1)}(q),\ldots,
 T_{i_N,m_N}^{(\ell_N)}(q)
 \right)^{\mathsf T}.
\end{equation}
For each $\alpha$, let
$E_\alpha=E_{\ell_\alpha,m_\alpha}(\cdot;q)$ be normalized in $L_d^2$, put
$T_\alpha=T_{i_\alpha,m_\alpha}^{(\ell_\alpha)}(q)$, and define
\[
 A_\alpha=\int_0^{T_\alpha}E_\alpha^2\dd\mu_d,\qquad
 B_\alpha=\int_{T_\alpha}^RE_\alpha^2\dd\mu_d,\qquad
 \Gamma_\alpha=T_\alpha^{d-1}E_\alpha'(T_\alpha)^2.
\]
The corresponding gradient is
\begin{equation}\label{eq:mixed-gradient}
 g_\alpha(r;q)=
 \frac{B_\alpha}{\Gamma_\alpha}E_\alpha(r)^2\ind_{(0,T_\alpha)}(r)
 -\frac{A_\alpha}{\Gamma_\alpha}E_\alpha(r)^2\ind_{(T_\alpha,R)}(r).
\end{equation}

\begin{theorem}
\label{thm:mixed-transversality}
Assume \eqref{eq:main-p-assumption}.  The map
$\mathbf T_{\mathfrak a}:L_d^p\to\mathbb R^N$ is continuously
Fr\'echet differentiable and
\begin{equation}\label{eq:mixed-derivative}
 D\mathbf T_{\mathfrak a}(q)h
 =\left(\int_0^R g_\alpha(r;q)h(r)\dd\mu_d(r)\right)_{\alpha=1}^N.
\end{equation}
At a fixed potential $q_*$, the following are equivalent:
\begin{enumerate}
\renewcommand{\labelenumi}{\textup{(\roman{enumi})}}
 \item $D\mathbf T_{\mathfrak a}(q_*)$ is surjective;
 \item $g_1(\cdot;q_*),\ldots,g_N(\cdot;q_*)$ are linearly independent
 in $L_d^{p'}$;
 \item there exist $h_1,\ldots,h_N\in L_d^p$ such that the matrix
 \begin{equation}\label{eq:mixed-test-matrix}
 M_{\alpha\beta}
 :=\int_0^R g_\alpha(r;q_*)h_\beta(r)\dd\mu_d(r)
 \end{equation}
 is nonsingular.
\end{enumerate}
If these conditions hold, they remain valid in a neighborhood of
$q_*$.  The mixed map is a submersion there, every sufficiently nearby
target vector is feasible, and every corresponding level set is a
$C^1$ Banach submanifold of codimension $N$.
\end{theorem}

\begin{proof}
Formula \eqref{eq:mixed-derivative} follows componentwise from
Theorem~\ref{thm:node-calculus}, applied in each fixed sector.  Surjectivity fails if and
only if a nonzero vector $c\in\mathbb R^N$ annihilates the range.  This
is equivalent to
$\sum_{\alpha=1}^Nc_\alpha g_\alpha=0$ in $L_d^{p'}$, proving the
equivalence of (i) and (ii).  If (i) holds, choose preimages $h_\beta$
of the coordinate vectors to obtain (iii).  Conversely, the
nonsingularity of \eqref{eq:mixed-test-matrix} forces the range to be
all of $\mathbb R^N$.  Continuity of the gradients makes the determinant
of $M(q)$ continuous, so transversality persists locally.  The
submersion theorem gives the remaining assertions.
\end{proof}

\begin{theorem}
\label{thm:mixed-local-optimization}
Let $q_0\in L_d^p$ satisfy the transversality conditions of
Theorem~\ref{thm:mixed-transversality}, and put
$\mathbf T_0=\mathbf T_{\mathfrak a}(q_0)$.  For every sufficiently
small $\boldsymbol\delta\in\mathbb R^N$, the problem
\begin{equation}\label{eq:mixed-optimization}
 \min\left\{\|q-q_0\|_{L_d^p}:
 \mathbf T_{\mathfrak a}(q)=\mathbf T_0+\boldsymbol\delta\right\}
\end{equation}
has a minimizer $\widehat q_{\boldsymbol\delta}$.  If it is nontrivial,
then for some multiplier vector $\boldsymbol\kappa\ne0$,
\begin{equation}\label{eq:mixed-KKT}
 \varphi_p(\widehat q_{\boldsymbol\delta}-q_0)
 =\sum_{\alpha=1}^N\kappa_\alpha
 g_\alpha(\cdot;\widehat q_{\boldsymbol\delta}).
\end{equation}
Equivalently, with
\begin{equation}\label{eq:mixed-Q}
 Q_{\boldsymbol\kappa}
 :=\varphi_{p'}\left(
 \sum_{\beta=1}^N\kappa_\beta
 g_\beta(\cdot;\widehat q_{\boldsymbol\delta})
 \right),
\end{equation}
set
$E_\alpha:=E_{\ell_\alpha,m_\alpha}
(\cdot;\widehat q_{\boldsymbol\delta})$ in the following equation.
Then each observed eigenfunction satisfies the coupled system
\begin{equation}\label{eq:mixed-coupled-system}
 -E_\alpha''-\frac{d-1}{r}E_\alpha'
 +\frac{\gamma_{\ell_\alpha}}{r^2}E_\alpha
 +(q_0+Q_{\boldsymbol\kappa})E_\alpha
 =\lambda_{\ell_\alpha,m_\alpha}
 (\widehat q_{\boldsymbol\delta})E_\alpha.
\end{equation}

Let $G_0=D\mathbf T_{\mathfrak a}(q_0)$ and let
$\mathscr R_0(\boldsymbol\delta)$ be the unique minimum-$L_d^p$-norm
solution of $G_0h=\boldsymbol\delta$.  Then every family of minimizers
satisfies
\begin{equation}\label{eq:mixed-local-reconstruction}
 \widehat q_{\boldsymbol\delta}-q_0
 =\mathscr R_0(\boldsymbol\delta)
 +o_{L_d^p}(|\boldsymbol\delta|).
\end{equation}
If $p=2$ and $d\in\{2,3\}$, the transversality condition is equivalent
to positive definiteness of the computable Gram matrix
\begin{equation}\label{eq:mixed-Gram}
 \mathbb G_{\mathfrak a,0}
 :=\left(\langle g_\alpha,g_\beta\rangle_{L_d^2}\right)_{\alpha,\beta=1}^N.
\end{equation}
In this case
\begin{equation}\label{eq:mixed-Hilbert-reconstruction}
 \mathscr R_0(\boldsymbol\delta)
 =G_0^*\mathbb G_{\mathfrak a,0}^{-1}\boldsymbol\delta,
\end{equation}
and, for every sufficiently small data vector, the global minimizer is
unique.
\end{theorem}
\begin{proof}
By Theorem~\ref{thm:mixed-transversality}, the mixed observation map is $C^1$
and has surjective derivative at $q_0$; its components are weakly sequentially
continuous by Theorem~\ref{thm:node-calculus}.  Hence
Proposition~\ref{prop:finite-observation-principle}, applied to
$\mathcal O=\mathbf T_{\mathfrak a}$, gives existence of the nearby global
minimizer and the expansion \eqref{eq:mixed-local-reconstruction}.
Surjectivity persists at the minimizer, so the Banach-space Lagrange
multiplier theorem gives \eqref{eq:mixed-KKT}; applying the inverse duality
map yields \eqref{eq:mixed-Q}--\eqref{eq:mixed-coupled-system}.

If $p=2$ and $d\in\{2,3\}$, transversality is equivalent to positive
definiteness of \eqref{eq:mixed-Gram}.  Proposition~\ref{prop:Hilbert-C2},
applied sectorwise, makes the mixed map $C^2$.  The Hilbert part of
Proposition~\ref{prop:finite-observation-principle} therefore gives
\eqref{eq:mixed-Hilbert-reconstruction} and uniqueness of the small-data
global minimizer.
\end{proof}

\subsection{Simultaneous spectral and nodal constraints}

For one fixed pair $(\ell,m)$ and
$\mathcal I=\{i_1<\cdots<i_N\}$, define the augmented observation map
\begin{equation}\label{eq:spectral-nodal-map}
 \mathcal F(q)=\mathcal F_{\ell,m,\mathcal I}(q)
 :=\left(
 \lambda_{\ell,m}(q),
 T_{i_1,m}^{(\ell)}(q),\ldots,
 T_{i_N,m}^{(\ell)}(q)
 \right)^{\mathsf T}\in\mathbb R^{N+1}.
\end{equation}
At a reference potential $q_0$, write
\begin{equation}\label{eq:augmented-linearization}
 \mathcal F_0:=\mathcal F(q_0),\qquad
 \mathcal G_0:=D\mathcal F(q_0):L_d^p\to\mathbb R^{N+1}.
\end{equation}
Let
\begin{equation}\label{eq:augmented-gradients}
 \psi_0:=E_{\ell,m}(\cdot;q_0)^2,
 \qquad
 \psi_\alpha:=g_{i_\alpha,m}^{(\ell)}(\cdot;q_0),
 \quad1\le\alpha\le N.
\end{equation}
Thus
\[
 \mathcal G_0h
 =\left(\int_0^R\psi_a(r)h(r)\dd\mu_d(r)\right)_{a=0}^N.
\]
For $\boldsymbol y\in\mathbb R^{N+1}$, define
\begin{equation}\label{eq:augmented-minimum-right-inverse}
 \rho_{\mathcal F,0}(\boldsymbol y)
 :=\min\{\|h\|_{L_d^p}:\mathcal G_0h=\boldsymbol y\},
 \qquad
 \mathscr R_{\mathcal F,0}(\boldsymbol y)
 :=\operatorname*{argmin}_{\mathcal G_0h=\boldsymbol y}
 \|h\|_{L_d^p}.
\end{equation}

\begin{theorem}\label{thm:spectral-nodal}
Assume \eqref{eq:main-p-assumption}.
\begin{enumerate}
\renewcommand{\labelenumi}{\textup{(\roman{enumi})}}
 \item The map \eqref{eq:spectral-nodal-map} is a $C^1$ submersion at every
 $q\in L_d^p$.

 \item For every $\Lambda\in\mathbb R$ and every compatible nodal target
 vector $\boldsymbol\tau$, the joint constraint set
 \begin{equation}\label{eq:spectral-nodal-constraint}
  \left\{q:\lambda_{\ell,m}(q)=\Lambda,
  \quad T_{i_\alpha,m}^{(\ell)}(q)=\tau_\alpha,
  \ 1\le\alpha\le N\right\}
 \end{equation}
 is nonempty, and the minimum-distance problem to any $q_0\in L_d^p$
 has a minimizer.  At every nontrivial minimizer there are multipliers
 $(\eta,\boldsymbol\kappa)\in\mathbb R\times\mathbb R^N$, not all zero,
 such that
 \begin{equation}\label{eq:spectral-nodal-KKT}
  \varphi_p(\widehat q-q_0)
  =\eta E_{\ell,m}(\cdot;\widehat q)^2
  +\sum_{\alpha=1}^N\kappa_\alpha
  g_{i_\alpha,m}^{(\ell)}(\cdot;\widehat q).
 \end{equation}

 \item Let $\boldsymbol y=(\xi,\boldsymbol\delta_T)
 \in\mathbb R^{N+1}$ tend to zero through data for which
 $\mathcal F_0+\boldsymbol y$ is compatible, where $\xi$ is the
 spectral displacement and $\boldsymbol\delta_T$ is the nodal displacement,
 and let $\widehat q_{\boldsymbol y}$ be any minimum-distance potential satisfying
 $\mathcal F(\widehat q_{\boldsymbol y})=\mathcal F_0+\boldsymbol y$.
 Then
 \begin{align}
  \widehat q_{\boldsymbol y}-q_0
  &=\mathscr R_{\mathcal F,0}(\boldsymbol y)
    +o_{L_d^p}(|\boldsymbol y|),
  \label{eq:augmented-local-reconstruction}\\
  \|\widehat q_{\boldsymbol y}-q_0\|_{L_d^p}
  &=\rho_{\mathcal F,0}(\boldsymbol y)+o(|\boldsymbol y|).
  \label{eq:augmented-local-distance}
 \end{align}

 \item If $p=2$ and $d\in\{2,3\}$, define the augmented Gram matrix
 \begin{equation}\label{eq:augmented-Gram}
  \mathbb M_0
  :=\big(\langle\psi_a,\psi_b\rangle_{L_d^2}\big)_{a,b=0}^N
  =\mathcal G_0\mathcal G_0^*.
 \end{equation}
 Then $\mathbb M_0$ is positive definite and, for all sufficiently small
 compatible $\boldsymbol y$, the minimizer is unique and depends $C^1$ on
 $\boldsymbol y$.  Moreover,
 \begin{align}
  \mathscr R_{\mathcal F,0}(\boldsymbol y)
  &=\mathcal G_0^*\mathbb M_0^{-1}\boldsymbol y,
  \label{eq:augmented-Hilbert-right-inverse}\\
  \widehat q_{\boldsymbol y}-q_0
  &=\mathcal G_0^*\mathbb M_0^{-1}\boldsymbol y
    +O_{L_d^2}(|\boldsymbol y|^2),
  \label{eq:augmented-Hilbert-reconstruction}\\
  \|\widehat q_{\boldsymbol y}-q_0\|_{L_d^2}^2
  &=\boldsymbol y^{\mathsf T}\mathbb M_0^{-1}\boldsymbol y
    +O(|\boldsymbol y|^3).
  \label{eq:augmented-Hilbert-cost}
 \end{align}
\end{enumerate}
\end{theorem}
\begin{proof}
At a potential $q$, the component gradients of $\mathcal F$ are
\[
 E^2,\qquad E^2H_1,\ldots,E^2H_N.
\]
If a linear combination vanishes, division by $E^2>0$ almost everywhere and
comparison of the distinct jumps of the $H_\alpha$ give
$c_1=\cdots=c_N=0$, and then $c_0=0$.  Hence $D\mathcal F(q)$ is surjective
for every $q$, proving (i).

For (ii), first realize the prescribed nodal vector by
Proposition~\ref{prop:multi-global-feasibility}.  If the resulting potential is
$q_{\boldsymbol\tau}$, set
\[
 \widetilde q
 =q_{\boldsymbol\tau}+\Lambda-\lambda_{\ell,m}(q_{\boldsymbol\tau}).
\]
A constant shift moves the eigenvalue by the same constant and leaves all
nodes unchanged, so every compatible joint target is feasible.  Weak
sequential continuity of the eigenvalue and nodal components then gives
existence by the direct method, while surjectivity gives the KKT relation
\eqref{eq:spectral-nodal-KKT} at every nontrivial minimizer.

For (iii), apply Proposition~\ref{prop:finite-observation-principle} to
$\mathcal O=\mathcal F$.  This yields
\eqref{eq:augmented-local-reconstruction}--\eqref{eq:augmented-local-distance}.
If $p=2$ and $d\in\{2,3\}$, the eigenvalue map is $C^2$ because its gradient
is $E^2$ and the smooth eigenfunction map takes values in
$\mathcal H_{\ell,0}\hookrightarrow L_d^4$; each nodal component is $C^2$ by
Proposition~\ref{prop:Hilbert-C2}.  The Hilbert clause of
Proposition~\ref{prop:finite-observation-principle} therefore yields the
positive-definite augmented Gram matrix, the inverse-Gram formula, local
uniqueness, and \eqref{eq:augmented-Hilbert-reconstruction}-\eqref{eq:augmented-Hilbert-cost}.
\end{proof}

\section{Conclusion and open problems}\label{sec:conclusion}

\begin{remark}
The global uniqueness theorem follows from three independent ingredients:
(i) global selection of the critical sign; (ii) uniqueness of positive
one-domain nonlinear branches; and (iii) strict opposite monotonicity of the
weighted masses.  Hence the rigidity is produced by the geometry of the
nodal constraint and not only by a local inverse mapping argument.
\end{remark}

We have established a singular Friedrichs finite-data variational framework directly
in an arbitrary fixed angular-momentum sector.  The formulation therefore
contains the radial $\ell=0$ profile without duplicating its spectral,
nodal, variational, or reconstruction statements.  For every fixed sector,
the Friedrichs eigenfunctions have the correct $r^\ell$ endpoint behavior,
nodal maps are weakly continuous and Fr\'echet differentiable, arbitrary
compatible same-mode nodal data are globally feasible, and minimum-distance
potentials exist.  The one-node and multi-node critical equations retain the
centrifugal term, while local reconstruction is governed by
minimum-norm right inverses and, in the Hilbert case, by positive definite
Gram matrices.

The energy identity makes clear which qualitative feature is genuinely
special: monotone radial dissipation occurs when $\ell=0$, whereas for
$\ell\geq1$ the centrifugal correction destroys a fixed sign.  The regular
shooting variable $W=r^{-\ell}U$ nevertheless gives a uniform
three-parameter representation for every fixed sector.

Relative to the weighted $\ell=0$ variational framework of Cheng, He, Wang,
and Xia \cite{ChengHeWangXia2026}, the additional conclusions established
here are the arbitrary-angular-momentum Friedrichs calculus, the conditional
mixed-sector theory, automatic same-mode spectral--nodal transversality, and
the global inward uniqueness theorem.  This is the precise sense in which the
present results extend, rather than reintroduce, the radial finite-data model.

For $d\ge2$, we also obtain a global uniqueness result for inward
displacements of the unique interior node of the second radial mode. For a constant prior,
the second radial mode, and $p>(d+2)/2$, the optimal potential is unique for
each $T_*\in(0,T_0)$.  The proof fixes the critical sign globally and
reduces the shooting system to the balance of a strictly decreasing
focusing-ball mass and a strictly increasing logistic-annulus mass.

For observations from different modes or angular momenta, local
identifiability is equivalent to surjectivity of the mixed derivative, or
to linear independence of the corresponding nodal gradients.  Under this
transversality condition we obtained local feasibility, existence, a coupled
critical system, reconstruction, and Hilbert-space local uniqueness.
One eigenvalue together with finitely many nodes of the same fixed-sector
eigenfunction is automatically transverse.  Every compatible joint target
is globally feasible, and the augmented observation admits sharp local
minimum-norm reconstruction, an inverse-Gram formula, and Hilbert-space
local uniqueness.

We conclude with several open problems:
\begin{enumerate}
\renewcommand{\labelenumi}{\textup{(\roman{enumi})}}
 \item global uniqueness outside the inward second-radial-mode regime of
 Theorem~\ref{thm:global-uniqueness-inward}, including outward nodes, higher modes, nonconstant
 priors, and higher angular momenta;
 \item monotonicity or nonsingularity of the fixed-sector shooting map;
 \item global feasibility for arbitrary collections of nodes from different
 modes or angular momentum sectors;
 \item generic transversality and quantitative lower bounds for mixed Gram
 matrices;
 \item simultaneous optimization with several independent eigenvalues and
 mixed nodal constraints;
 \item sharp endpoint blow-up rates and stable reconstruction from noisy or
 interval-valued data.
\end{enumerate}
Due to the length limitation of
the paper, we will consider the above problems as future work. Furthermore, this work provides some
insights for future developments and research of Schr\"{o}dinger operators in physical applications.\\

\textbf{Declaration}

The authors declare that they have no known competing financial interests or personal relationships
that could have appeared to influence the work reported in this paper.


\begin{thebibliography}{99}


\bibitem{BrezisOswald1986}
H.~Brezis and L.~Oswald,
Remarks on sublinear elliptic equations,
\emph{Nonlinear Anal.} \textbf{10} (1986), 55--64.

\bibitem{AlbeverioHrynivMykytyuk2007}
S. Albeverio, R. Hryniv, and Ya. Mykytyuk,
Inverse spectral problems for Bessel operators,
\emph{J. Differential Equations} \textbf{241} (2007), 130--159.

\bibitem{ChengHeWangXia2026}
Z. Cheng, Y. He, S. Wang, and Y. Xia,
Finite inverse nodal problems for singular weighted Sturm--Liouville
equations: variational selection and spectral matching,
preprint, arXiv:2609.16481 (2026).

\bibitem{Chen-Cheng11}
X. Chen, Y. Cheng, C. Law,
Reconstructing potentials from zeros of one eigenfunction,
\emph{Trans. Amer. Math. Soc.} \textbf{363} (2011), 4831--4851.


\bibitem{ChuMengWangZhang2024}
J. Chu, G. Meng, F. Wang, and M. Zhang,
Optimization problems on nodes of Sturm--Liouville operators with $L^p$ potentials,
\emph{Math. Ann.} \textbf{390} (2024), 1401--1417.

\bibitem{ChuMengWangZhang2025}
J. Chu, G. Meng, F. Wang, and M. Zhang,
Complete continuity and Fr\'echet derivatives of nodes in potentials for one-dimensional $p$-Laplacian,
\emph{J. Differential Equations} \textbf{416} (2025), 1960--1976.

\bibitem{GidasNiNirenberg1979}
B. Gidas, W.-M. Ni, and L. Nirenberg,
Symmetry and related properties via the maximum principle,
\emph{Comm. Math. Phys.} \textbf{68} (1979), 209--243.

\bibitem{GuoZhang2022}
S. Guo and M. Zhang,
On the dependence of nodes of Sturm--Liouville problems on potentials,
\emph{Mediterr. J. Math.} \textbf{19} (2022), Paper No. 168.

\bibitem{HaldMcLaughlin1989}
O.~H. Hald and J.~R. McLaughlin,
Solutions of inverse nodal problems,
\emph{Inverse Problems} \textbf{5} (1989), 307--347.

\bibitem{HeWuXiaZhang2025}
Y. He, M. Wu, Y. Xia, and M. Zhang,
A novel and application-oriented inverse nodal problem for Sturm--Liouville operators,
\emph{Math. Ann.} \textbf{393} (2025), 3119--3140.

\bibitem{KostenkoSakhnovichTeschl2010}
A. Kostenko, A. Sakhnovich, and G. Teschl,
Inverse eigenvalue problems for perturbed spherical Schr\"odinger operators,
\emph{Inverse Problems} \textbf{26} (2010), 105013, 14 pp.

\bibitem{KostenkoTeschl2011}
A. Kostenko and G. Teschl,
On the singular Weyl--Titchmarsh function of perturbed spherical Schr\"odinger operators,
\emph{J. Differential Equations} \textbf{250} (2011), 3701--3739.

\bibitem{McLaughlin1988}
J.~R. McLaughlin,
Inverse spectral theory using nodal points as data---a uniqueness result,
\emph{J. Differential Equations} \textbf{73} (1988), 354--362.


\bibitem{NorisTavaresVerzini2014}

B. Noris, H. Tavares, and G. Verzini,
Existence and orbital stability of the ground states with prescribed mass for
the $L^2$-critical and supercritical NLS on bounded domains,
\emph{Anal. PDE} \textbf{7} (2014), no.~8, 1807--1838.

\bibitem{LandauLifshitz1977}
L.~D. Landau and E.~M. Lifshitz,
\emph{Quantum Mechanics: Non-Relativistic Theory}, 3rd ed.,
Course of Theoretical Physics, Vol.~3, Pergamon Press, Oxford, 1977.

\bibitem{ChadanSabatier1989}
K. Chadan and P.~C. Sabatier,
\emph{Inverse Problems in Quantum Scattering Theory}, 2nd ed.,
Springer-Verlag, Berlin, 1989.

\bibitem{GharaatiKhordad2010}
A. Gharaati and R. Khordad,
A new confinement potential in spherical quantum dots: modified Gaussian potential,
\emph{Superlattices Microstruct.} \textbf{48} (2010), 276--287.

\bibitem{Serier2007}
F. Serier,
The inverse spectral problem for radial Schr\"odinger operators on $[0,1]$,
\emph{J. Differential Equations} \textbf{235} (2007), 101--126.


\bibitem{Teschl2014}
G. Teschl,
\emph{Mathematical Methods in Quantum Mechanics: With Applications to Schr\"odinger Operators}, 2nd ed.,
Graduate Studies in Mathematics 157, American Mathematical Society, Providence, RI, 2014.

\bibitem{Yang1997}
X.-F. Yang,
A solution of the inverse nodal problem,
\emph{Inverse Problems} \textbf{13} (1997), 203--213.

\bibitem{XuYangBondarenko2023}
X.-J. Xu, C.-F. Yang, and N. Bondarenko,
Inverse spectral problems for radial Schr\"odinger operators and closed systems,
\emph{J. Differential Equations} \textbf{342} (2023), 343--368.

\bibitem{Zettl}
A. Zettl,
\emph{Sturm--Liouville Theory},
Mathematical Surveys and Monographs 121, American Mathematical Society, Providence, RI, 2005.


\bibitem{ChuMengXie2026}
J. Chu, G. Meng, and N. Xie,
An inverse problem for Sturm--Liouville equations with a fixed node,
\emph{J. Differential Equations} \textbf{454} (2026), 113966.

\bibitem{ArslantasDurakAmirov2026}
M. Arslanta\c{s}, S. Durak, and R. Amirov,
Inverse nodal problem for singular Sturm--Liouville operator,
\emph{Math. Methods Appl. Sci.} \textbf{49} (2026), no.~9, 9022--9032.

\bibitem{JiangXuYang2026}
C.-T. Jiang, X.-J. Xu, and C.-F. Yang,
Solving inverse nodal problems of Sturm--Liouville operator with a point
interaction based on Legendre wavelet bases,
\emph{Math. Comput. Simulation} \textbf{250} (2026), 1313--1324.

\bibitem{PanakhovKoyunbakan2006}
E.~S. Panakhov and H. Koyunbakan,
Inverse nodal problems for second order differential operators with a regular singularity,
\emph{Int. J. Difference Equ.} \textbf{1} (2006), no.~2, 241--247.

\bibitem{KoyunbakanPanakhov2006}
H. Koyunbakan and E.~S. Panakhov,
Solution of a discontinuous inverse nodal problem on a finite interval,
\emph{Math. Comput. Modelling} \textbf{44} (2006), 204--209.

\bibitem{KoyunbakanPanakhov2007}
H. Koyunbakan and E.~S. Panakhov,
A uniqueness theorem for inverse nodal problem,
\emph{Inverse Probl. Sci. Eng.} \textbf{15} (2007), no.~6, 517--524.

\bibitem{guo-wei}
Y. Guo and G. Wei,
The sharp conditions of the uniqueness for inverse nodal problems,
\emph{J. Differential Equations} \textbf{266} (2019), 4432--4449.

\bibitem{G-Z2}
S. Guo and M. Zhang,
A variational approach to the optimal locations of the nodes of the second Dirichlet eigenfunctions,
\emph{Math. Methods Appl. Sci.} \textbf{46} (2023), 11983--12006.

\bibitem{H-M}
O. Hald and J. McLaughlin,
Inverse nodal problems: finding the potential from nodal lines,
\emph{Mem. Amer. Math. Soc.} \textbf{119} (1996), 148 pp.

\bibitem{V-I2019}
Y. Il'yasov and N. Valeev,
On nonlinear boundary value problem corresponding to $N$-dimensional inverse spectral problem,
\emph{J. Differential Equations} \textbf{266} (2019), 4533--4543.

\bibitem{P-S}
J. Pinasco and C. Scarola,
A nodal inverse problem for a quasi-linear ordinary differential equation in the half-line,
\emph{J. Differential Equations} \textbf{261} (2016), 1000--1016.

\bibitem{w-y}
Y. Wang and V. Yurko,
On the inverse nodal problems for discontinuous Sturm--Liouville operators,
\emph{J. Differential Equations} \textbf{260} (2016), 4086--4109.

\bibitem{YANGCF}
C.-F. Yang,
Inverse nodal problems of discontinuous Sturm--Liouville operator,
\emph{J. Differential Equations} \textbf{254} (2013), 1992--2014.

\bibitem{yang2}
X.-F. Yang,
A new inverse nodal problem,
\emph{J. Differential Equations} \textbf{169} (2001), 633--653.

\bibitem{zhang}
M. Zhang,
Continuity in weak topology: higher order linear systems of ODE,
\emph{Sci. China Ser. A} \textbf{51} (2008), 1036--1058.

\end{thebibliography}
\end{document}